\documentclass[12pt, letterpaper]{amsart}

\RequirePackage{etex}
\usepackage{amsmath,amssymb,amsthm,amsfonts}
\usepackage{mathtools}
\usepackage{mathrsfs}
\usepackage[utf8]{inputenc}
\usepackage{graphicx}
\usepackage{caption}

\usepackage{hyperref}
\usepackage{tikz}
\usepackage[all]{xy}
\usepackage{fullpage}

\usepackage{comment}
\usepackage{etex}
\usepackage{array}
\usepackage{graphics}
\usepackage{setspace}

\usepackage{amscd}
\usetikzlibrary{matrix,arrows,decorations.pathmorphing}
\tikzset{my loop/.style =  {to path={
  \pgfextra{}
  [looseness=12,min distance=10mm]
  \tikz@to@curve@path},font=\sffamily\small
  }}  
\usepackage{tikz-cd}

\hypersetup{    colorlinks=true,    linkcolor=blue,    filecolor=magenta,          urlcolor=cyan,}

\theoremstyle{plain}
\newtheorem{theorem}{Theorem}[section]

\newtheorem{lemma}[theorem]{Lemma}
\newtheorem{corollary}[theorem]{Corollary}
\newtheorem{proposition}[theorem]{Proposition}

\theoremstyle{definition}
\newtheorem{remark}[theorem]{Remark}

\newtheorem{definition}[theorem]{Definition}
\newtheorem{question}[theorem]{Question}

\DeclareMathOperator{\SL}{SL}

\DeclareMathOperator{\Hom}{Hom}
\DeclareMathOperator{\Aut}{Aut}

\newcommand{\C}{\mathbb{C}}

\newcommand{\s}{\mathfrak{s}}

\title{
Eigenperiods and the moduli of points in the line
}
\author{Haohua Deng}
\address[H.D.]{Department of mathematics, Duke University, Durham, NC, 27708}
\email{haohua.deng@duke.edu}

\author{Patricio Gallardo}
\address[P.G.]{Department of mathematics, University of California at Riverside, Riverside, CA, 92521}
\email{pgallard@ucr.edu}

\begin{document}
\begin{abstract}
We study the period map of configurations of $n$ points on the projective line constructed via a cyclic cover branching along these points. By considering the decomposition of its Hodge structure into eigenspaces, we establish the codimension of the locus where the eigenperiod map is still pure. Furthermore, we show that the period map extends to the divisors of a specific moduli space of weighted stable rational curves, and that this extension satisfies a local Torelli map along its fibers.
\end{abstract}
\maketitle
\section{Introduction}
Understanding period maps and Torelli-type results that relate geometric and Hodge-theoretic compactifications of moduli spaces of varieties is a central problem within algebraic geometry.  A classic example of such results is due to Deligne and Mostow \cite{DM86}. They showed that certain geometric compactifications of the moduli space for up to 12 partially labeled points in $\mathbb{P}^1$, constructed via Geometric Invariant Theory (GIT), are isomorphic to a Hodge-theoretic compactification of the image of their period domain - the Baily-Borel compactification. 
Recently, the Deligne-Mostow isomorphisms have gained much attention, with efforts to extend such isomorphisms to other geometric and Hodge-theoretic  compactifications of such moduli spaces of points \cite{gallardo2021geometric},
to understand their birational geometry
\cite{HM22}, \cite{hulek2024compactifications},  
as well as applications to other moduli spaces, such as curves and cubic surfaces \cite{Loo24}, \cite{casalaina2023birational}.

Yet, whether the Deligne-Mostow results can be partially generalized to an arbitrary number of points remains unknown. To understand the difficulties, let us recall the construction of the period map for $n$ points in $\mathbb{P}^1$. It begins with a cyclic cover $C$ of degree $d$ branching along these points. The Galois group of this finite cover lifts to both the cohomology of and the Hodge structure on $C$. Thus, the action of the Galois group induces an eigenspace decomposition on $H^1(C)$ indexed by an integer $k$.  One can restrict the Hodge structure of $C$ to the $k$-th eigenspace to construct a potential eigenperiod map characterized by the data $(d,k)$. Two important problems arise: for an arbitrary number of points the potential eigenperiod map is not dominant and the monodromy representation may not even be discrete. Nowadays, it is understood that many interesting period maps are dominant into their period domains.  McMullen described conditions on $(d,k)$ ensuring that the associated monodromy representation is discrete, see Theorem \ref{thm:Mcmullen} and Table \ref{tab:ancestorCases}. Therefore, it is possible to define a family of eigenperiod maps that contain and generalize the Deligne-Mostow cases, as seen in Proposition \ref{prop:EigenPeriodM0n}.

This paper is devoted to studying two partial extensions of eigenperiod maps mentioned above. 
This is possible because there is a geometric smooth normal crossing compactification of the moduli of $n$ labeled points on the line known as $\overline{M}_{0,n}$. The first partial extension comes from the classical fact in Hodge theory that we may extend the eigenperiod map to boundary divisors with finite monodromy. The second partial extension is known as Kato--Usui's compactification
\cite{KU08}, which aligns with the toroidal compactification introduced by \cite{AMRT10} in the classical case.

We begin with a proposition that should be clear to an expert in Hodge theory. 

\begin{proposition}
\label{prop:extension}
Let $n \geq 5$ and $k/d$ be as specified in Table \ref{tab:ancestorCases}, and suppose that $ \mathrm{gcd}(d,n-1)=1$ holds. There exists an eigenperiod map coming from the family of cyclic cover of $\mathbb{P}^1$ branched over marked points, see Proposition \ref{prop:EigenPeriodM0n},      
\[
\Phi_{n,d,k}:M_{0,n} \to
\Gamma_{n,d,k}\backslash D_{n,d,k}
\]
such that 
\begin{itemize}
    \item[(i)]  There exists a maximal partial compactification $M_{n,d,k}^{\mathrm{pr}} \subset \overline{M}_{0,n}$ over which $\Phi_{n,d,k}$ extends, that is 
    \[
    \Phi_{n,d,k}:M_{n,d,k}^{\mathrm{pr}} \to
    \Gamma_{n,d,k}\backslash D_{n,d,k}.
    \]
    \item[(ii)]
    The space $M_{n,d,k}^{\mathrm{pr}}$ is characterized by the following property: For any $p\in \overline{M}_{0,n}$, $p\in M_{n,d,k}^{\mathrm{pr}}$ if and only if for any $1$-dimensional contractible neighboorhood $p\in U\subset \overline{M}_{0,n}$ with $U-\{p\}\subset M_{0,n}$, the variation of Hodge structure associated to $\Phi_{n,d,k}|_{U-\{p\}}$ has finite monodromy. 
    
\end{itemize}
\end{proposition}
We remark the necessity of the condition $\mathrm{gcd}(d, n-1)=1$: It secures the existence of a branched covering map $\mathcal{C}\rightarrow \mathbb{P}^1$ of generic degree $d$, and the ramification loci is the union of $n$ distinct points on $\mathbb{P}^1$, each of which has ramification degree $d-1$. Such a branched covering map does not exist if $\mathrm{gcd}(d, n-1)>1$, as by projective automorphism we may arrange one of the $n$ points to be the infinity and each of the $n-1$ finite branched points has ramification degree $d-1$, then the infinity will have ramification degree less than $d-1$. See \cite[Sec. 3]{McM13}.

The locus $M_{n,d,k}^{\mathrm{pr}}$ plays important roles. From the moduli point of view, a point $p\in \overline{M}_{0,n}$ gives a stable curve in $\overline{M}_{g}$ via the cyclic $d$-cover, see \cite{fedorchuk2011cyclic}. The resultant family of stable curves over $\overline{M}_{0,n}$ induces a moduli map
\begin{center}
    $\mathcal{F} : \overline{M}_{0,n}\rightarrow \overline{M}_{g}$.
\end{center}
The locus on $\overline{M}_{0,n}$ whose image under $\mathcal{F}$ produces stable curves of compact type, i.e., whose dual graph is a tree, is exactly $\cap_k M_{n,d,k}^{\mathrm{pr}}$. The spaces $M_{n,d,k}^{\mathrm{pr}}$ are also relevant to Hodge theory because the limiting mixed Hodge structures (LMHS) obtained from the local period maps $\Phi_{n,d,k}$ over $\Delta$ around $p \in M_{n,d,k}^{\mathrm{pr}}$ are pure as a direct consequence of finite local monodromy. 
We say that
$M_{n,d,k}^{\mathrm{pr}}$ is the $(d,k)$-pure locus, or simply pure locus if no confusion were caused.

In our first main Theorem below, we will give precise combinatorial conditions for a point $p\in \overline{M}_{0,n}$ to locate in $M_{n,d,k}^{\mathrm{pr}}$.  
Additionally, we calculate the codimension of the complement of the ``pure" loci as established in Proposition  \ref{prop:extension} within a pertinent GIT compactification.

For the next statement recall that the $\overline{M}_{0,n}$ is a smooth normal crossing compactification of $M_{0,n}$ whose irreducible divisors are labeled as $D_{I}$ with $I \subset \{1, \ldots, n\}$ with $2\leq |I|\leq \lfloor\frac{n}{2}\rfloor$. A generic point in $D_I$ parametrizes a stable curve with two irreducible components supporting the points $p_i$ with $i \in I$ and $p_{j}$ with $j \in I^c$ - the complement of $I$.
\begin{theorem}\label{thm:codimensions}
Let $M_{n,d,k}^{\mathrm{pr}}$ be as in Proposition \ref{prop:extension}, then 
\begin{align*}
M_{n,d,k}^{\mathrm{pr}} = 
\overline{M}_{0,n}
\setminus
\bigcup_{ I \notin \mathcal{I}^{pr} } 
D_{I}
&& 
\mathcal{I}^{pr} := 
\{ I \; : \;
k\min\{ |I|, |I^c| \} \not\equiv 0 \ \text{mod} \ (d) 
\; \}\subset 2^{\mathcal{I}}.
\end{align*}

\noindent Moreover, let
$
\overline{M}^{\mathrm{\mathrm{\mathrm{GIT}}}}_{0, \mathbf{w}} $ be the GIT compactification relative to the weights
$ \mathbf{w} = (\frac{2}{n}, ..., \frac{2}{n})$
if $d|n$, and 
$
\mathbf{w} = 
(\frac{1}{n-1}+\epsilon, ..., \frac{1}{n-1}+\epsilon, 1-(n-1)\epsilon )
$
if  $d \nmid n$ and 
$\mathrm{gcd}(d,n-1)=1$. 
Let  
$
\pi: \overline{M}_{0,n} \longrightarrow \overline{M}^{\mathrm{\mathrm{\mathrm{GIT}}}}_{0, \mathbf{w}} 
$ be the reduction morphism, define
\begin{align*}
H(n,d,k):=\mathrm{codim}\left( 
\overline{M}^{\mathrm{\mathrm{\mathrm{GIT}}}}_{0, \mathbf{w}} 
\setminus
\pi(M_{n,d,k}^{\mathrm{pr}})
\right),
\end{align*}
then it holds that for $d|n$, 
\begin{align*}
H(n,d,k) = \left\{ \begin{array}{rcl}
 \frac{d}{\mathrm{gcd}(k,d)}-1 & \mbox{\textrm{if}}
& \textrm{gcd}(k,d)>\frac{2d}{n} \\ n-3 & \mbox{\textrm{if}} & \textrm{gcd}(k,d)=\frac{2d}{n} \\ \infty & \mbox{\textrm{if}} & \textrm{gcd}(k,d)<\frac{2d}{n}
\\
\end{array}\right..    
\end{align*}
On the other hand, for $d\nmid n$ and $\mathrm{gcd}(d,n-1)=1$, 
\begin{align*}
H(n,d,k) = \left\{ \begin{array}{rcl}
 \frac{d}{\mathrm{gcd}(k,d)}-1 & \mbox{\textrm{if}}
& \textrm{gcd}(k,d)>\frac{d}{n-2} \\ n-3 & \mbox{\textrm{if}} & \textrm{gcd}(k,d)=\frac{d}{n-2} \\ \infty & \mbox{\textrm{if}} & \textrm{gcd}(k,d)<\frac{d}{n-2}
\\
\end{array}\right.    
\end{align*}
Note that if the codimension is infinite, it implies that $\pi(M_{n,d,k}^{\mathrm{pr}})=\overline{M}^{\mathrm{\mathrm{\mathrm{GIT}}}}_{0, \mathbf{w}}$.
\end{theorem}
To prove Theorem \ref{thm:codimensions}, we analyze in Section ~\ref{sec:PureLocus} the eigenspectra of the degenerations parametrized by the partial compactification $M_{n,d,k}^{\mathrm{pr}}$. This perspective extends Steenbrink's theory on the spectrum of hypersurface singularities, as seen in \cite{KL20}, \cite{castor2022remarks}. The proof is then completed by studying  the behavior of the LMHS arising from degenerations of eigenperiod maps, as detailed in Section
\ref{sec:ExtensionPureLocus}. 

To describe our next main result, we recall that by Proposition \ref{prop:extension} and Kato-Usui theory \cite{KU08},  there exists a finite-index normal neat subgroup $\Gamma^{\mathrm{nm}}_{n,d,k}\leq \Gamma_{n,d,k}$, and a normalization of a finite cover $\overline{M}^{\mathrm{nm}}_{0,n} \to \overline{M}_{0,n}$ such that the lift of $\Phi_{n,d,k}$ admits a Kato-Usui-Toroidal type extension:
\begin{align*}
    \xymatrix{
    \Phi^{\mathrm{nm}}_{n,d,k}:   \overline{M}^{\mathrm{nm}}_{0,n}
    \ar@{.>}[r]& 
    \left( \Gamma^{\mathrm{nm}}_{n,d,k}\backslash D_{n,d,k} \right) 
    \bigsqcup_{N\in \Gamma^{\mathrm{nm}}_{n,d,k}\backslash \Sigma}(\Gamma_N\backslash B(N)).
    }
\end{align*}
in codimension one (Here, $\Gamma_N$ denotes the stabilizer of $N$ in $\Gamma^{\mathrm{nm}}_{n,d,k}$). Therefore, it is most interesting to understand the extended period map better.

Instead of using $\overline{M}_{0,n}$, we recall that there is a smooth geometric compactification for $n \geq  4$, commonly known as Hassett's moduli spaces of weighted stable curves \cite{hassett2003moduli}. For the particular case of $n =  2r$, there is a unique divisor $\mathcal{D}$ in  
$\overline{M}_{0,\mathbf{\frac{2}{n}} + \epsilon}/S_n$ that generically parametrizes a weighted stable curve described as follows: 
It is the union of two $\mathbb{P}^1$ components, each with $r$ distinct marked points glued at another nodal singularity. 
In an open locus $\mathcal{D}^{\circ} \subset \mathcal{D}$ on which the gluing point does not collide with any marked point on either component, we obtain a reduction map 
\begin{align*}
\mathcal{R}:& \mathcal{D}^{\circ} \longrightarrow \mathrm{Sym}^2(M_{0,r}/S_r)
\end{align*}
by disregarding the gluing points on each $\mathbb{P}^1$, see Equation \eqref{eq:R}.  By
the same argument as used for Proposition \ref{prop:extension}, we have a rational map 
$$
\xymatrix{
\overline{M}_{0,\mathbf{\frac{2}{n}} + \epsilon}^{nm}
\ar@{.>}[r]& 
(\Gamma^{\mathrm{nm}}_{n,d,k}\backslash D_{n,d,k})_{\Sigma}
}
$$
where $\overline{M}_{0,\mathbf{\frac{2}{n}} + \epsilon}^{nm}$ is a finite cover of $\overline{M}_{0,\mathbf{\frac{2}{n}} + \epsilon}/S_n$.  Our next result extends this rational map to the interior of the divisors $\mathcal{D}$, and shows that such an extension satisfies a local Torelli-type theorem. 
We denote the union of $M_{0,n}$ and all codimension one strata within $\overline{M}_{0,\mathbf{\frac{2}{n}} + \epsilon}$ 
as $\overline{M}^{\circ}_{0,\mathbf{\frac{2}{n}} + \epsilon} \subset \overline{M}_{0,\mathbf{\frac{2}{n}} + \epsilon}$. 
The selection of a neat subgroup $\Gamma^{\mathrm{nm}}_{n,d,k} \subset 
\Gamma_{n,d,k}$ induces a lift $\mathcal{D}^{\mathrm{nm},\circ}$ and $\mathcal{R}^{\mathrm{nm}}$ for both the divisor $\mathcal{D}^{\circ}$ and the map $\mathcal{R}$. 
In our case, $\Sigma$ is a polyhedral fan in $\mathrm{Lie}(U(r,s))$ with rays generated by the monodromy logarithm $N$ around each divisor $\mathcal{D}$. 
\begin{theorem}\label{thm:KU}
Let $(n,d,k)$ be as in Proposition \ref{prop:extension}(i) and with $n$ even. Then, there exists a finite-index normal neat subgroup $\Gamma^{\mathrm{nm}}_{n,d,k}\leq \Gamma_{n,d,k}$, and a normalization of a finite cover $\overline{M}^{\mathrm{nm},\circ}_{0,\mathbf{\frac{2}{n}} + \epsilon} \longrightarrow \overline{M}_{0,\mathbf{\frac{2}{n}} + \epsilon}^{\circ}/S_n$ such that:
\begin{enumerate}
    \item  The eigenperiod map
    $$
    \phi_{n,d,k} :M_{0,n}/S_n \longrightarrow 
    \Gamma_{n,d,k}\backslash D_{n,d,k}
    $$
    has a unique lift which admits a Kato-Usui type extension:
    \begin{align*}
    \xymatrix{
    \phi^{\circ}_{n,d,k}:   \overline{M}^{\mathrm{nm},\circ}_{0,\mathbf{\frac{2}{n}} + \epsilon}
    \ar[r]& 
    \left( \Gamma^{\mathrm{nm}}_{n,d,k}\backslash D_{n,d,k} \right) 
    \bigsqcup_{N\in \Gamma^{\mathrm{nm}}_{n,d,k}\backslash \Sigma}(\Gamma_N\backslash B(N))
    }
    \end{align*}
which is regular, where $\Gamma_N$ denotes the stabilizer of $N$ in $\Gamma^{\mathrm{nm}}_{n,d,k}$.
    \item For $k=1$, the restriction of $\phi^{\circ}_{n,d,k}$ to any fiber $\left( \mathcal{R}^{\mathrm{nm}} \right)^{-1}(x)$ is an immersion.
    i.e. the map is locally injective. 
\end{enumerate}
\end{theorem}
Our result includes all the Deligne-Mostow cases \cite{gallardo2021geometric} and some cases with $n>12$ that are not covered by Deligne-Mostow. The first statement is a consequence of the techniques employed in \cite[Sec. 3.3]{gallardo2021geometric}. The second part of the theorem follows from a study of the associated Abel-Jacobi map.

\subsection*{Acknowledgements}
We would like to thank Matt Kerr, Gregory Pearlstein, Nolan Schock, and Yilong Zhang for their insightful discussions and comments. We also thank the anonymous referee for careful and insightful comments that significantly improved the article. We are grateful for the working environments at the Department of Mathematics at Washington University in St. Louis, University of California, Riverside, and  Duke University, where this research was conducted.  
The second author is partially supported by the National Science Foundation under Grant No. DMS-2316749.

\setcounter{tocdepth}{2}
\tableofcontents

\section{Preliminaries}

\subsection{Geometric compactifications}
\label{sec:GeometricCompactifications}
We describe two families of compactifications for the moduli space $M_{0,n}$ of $n$ labelled points in the projective line: The one constructed via Geometric Invariant Theory (GIT) and the one parametrizing stable pairs as defined in the Minimal Model Program (MMP). 
We refer to them as geometric compactifications because every point of such compactifications parametrizes configurations of points in a curve.  We will start with the GIT compactification. We refer the reader to \cite{mumford1994geometric} for the general theory.
\begin{theorem}\cite[Chapter 4.4]{mumford1994geometric}
\label{thm:GITcompactification}
Given a weight vector $\mathbf{w}:=(w_1, \ldots, w_n) \in \mathbb{Q}^n$ such that $1 \geq w_i >0$ and $w_1+ \cdots + w_n=2$, 
there is an irreducible projective variety  
\begin{align*}
\overline{M}_{0,\mathbf{w}}^{\mathrm{GIT}}
:=
(\mathbb{P}^1)^n / \!\! /_{\mathbf{w}} \SL_2
\end{align*}
known as the GIT compactification that satisfies the following properties: 
\begin{itemize}
    \item It has a Zariski open subset known as the stable locus, that  contains $M_{0,n}$, and 
    parametrizes configurations  $(p_1, \ldots, p_n)$ of $n$ labelled points in $\mathbb{P}^1 $ such that  
    $\sum_{p_i =p} w_i < 1$ for all $p \in \mathbb{P}^1$.
    \item  For even $n$ and weights $w_i = \frac{2}{n}$ for all $i$, the complement of the stable locus is a finite collection of points. Each of such points is associated to a degeneration where $p_{i_1} = \cdots = p_{i_{\frac{n}{2}}}$ and 
    $p_{i_{\frac{n}{2}+1}}= \cdots =  p_{i_n}$.
\end{itemize}
\end{theorem}
\begin{remark}
In general, the complement of the stable locus in the GIT quotient is known as 
the GIT boundary.  Every point of the stable locus is associated with a unique $\SL_2$-orbit (that is a configuration of points up to isomorphism.) This is not true for the GIT boundary where several non-isomorphic orbits maybe associated with the same point in the quotient.
\end{remark}

We need the following moduli of partially labeled points to use results from  \cite[Sec 2]{McM13}. If $\mathbb{C}^{\langle n \rangle} $ is the space of monic complex polynomials of degree $n$ with non-vanishing discriminant,
then the moduli space of $n$ distinct non-labelled points in the complex line up to 
$\mathrm{Aut}(\mathbb C)$ is given by
\begin{align}\label{eq:PointsAffineLine}
M_{0,n}^* := \mathbb{C}^{\langle n \rangle} /
\mathrm{Aut}(\mathbb C)
&&
\mathrm{Aut}(\mathbb C):= \{ (a,b) 
\in \mathbb{G}_m \rtimes \mathbb{C}
\; | \; z \mapsto az +b\}.
\end{align}
A labelled version of this space  is known as $T_{1,n}$ and it is was introduced for arbitrary dimension in  \cite{chen2009pointed}.
Let $T_{1,n}^{\circ} \subset T_{1,n}$ be the open parametrizing configuration of $n$ distinct labelled points. 
By \cite[Prop 3.4.3]{chen2009pointed}, we know that $T_{1,n}^{\circ} \cong M_{0,n+1}$
with the map  $\mathbb{C} \to \mathbb{P}^1$
induced by  $z_i \mapsto [1:z_i]$ and the $(n+1)$ point being $[0:1]$. 
Therefore, by their definition we have
$$
M_{0,n}^* \cong T_{1,n}^{\circ}/S_n
\cong
M_{0,n+1}/S_n
$$
To obtain a GIT compactification, let 
$\mathbf{v} = (v_1, \ldots, v_n, v_{n+1}) \in \mathbb{Q}^{n+1}$ be a vector of weights such that $1 \geq v_i > 0$ and $v_1 + \ldots + v_{n+1}=2$. If we suppose that $v_1=\ldots = v_n$, then we have the diagram:
\begin{align*}
\xymatrix{
M_{0,n+1} \ar[d] \ar[r] &  \overline{M}^{\mathrm{GIT}}_{\mathbf{v}} \ar[d]
 \\
M_{0,n}^* \ar[r] & \overline{M}^{\mathrm{GIT}}_{\mathbf{v}}/S_n
  }  
\end{align*}
The conditions on the weights imply they can be written as 
\begin{align*}
v_1 = \cdots = v_n = a
&&
v_{n+1} = 2-na
\text{ with $a > 0$ }
\end{align*}
From all possible choices, we highlight the case $v_{n+1}=1-n\epsilon$ and $v_1=...=v_n=\frac{1}{n} + \epsilon$
with  $\frac{1}{n(n-1)}  > \epsilon$.
For the same argument that in \cite[Lemma 4.13]{gallardo2017wonderful},  we have that 
$\overline{M}^{\mathrm{GIT}}_{\mathbf{v}} \cong \mathbb{P}^{n-2}$ and the GIT stable locus with respect to this choice of weights will parametrize configurations where any proper subset of the first $n$ points can collide within each other but never with the point at infinity.

Next, we describe the stable pair compactifications. They were introduced by Hassett in \cite{hassett2003moduli} for arbitrary genus. However, we will focus in the genus zero case.  Let $n$ be a positive integer and let  $\mathbf{b} \in \mathbb{Q}^n$ be a rational weight vector such that $2 < \sum b_i$, and $0 < b_i \leq 1$.
\begin{definition}\cite{hassett2003moduli}
A weighted stable rational curve for the weights $\mathbf{b} = (b_1, \ldots, b_n)$ is a pair
$\left( C, \sum_{i=1}^n b_ix_i \right)$ such that
\begin{itemize}
    \item $C$ is at worst nodal connected projective curve
of arithmetic genus zero,
\item $x_i$ are smooth points of $C$,
    \item If $x_{i_1} = \cdots = x_{i_k}$, then $b_{i_1} + \cdots b_{i_k} \leq 1$
    \item The divisor $K_C + \sum_{i=1}^n b_ix_i$ is ample.
\end{itemize}
\end{definition}
\begin{theorem}\cite[Theorem 2.1]{hassett2003moduli}
There exist a finite smooth projective moduli space 
$\overline{M}_{0,\mathbf{b}}$ parametrizing weighted stable rational curves of weight $\mathbf{b}$, and containing $M_{0,n}$ as a Zariski open locus. For $\mathbf{b} = (1, \ldots, 1)$, this space is the simple normal crossing compactification of $M_{0,n}$
known as the Deligne-Mumford-Knudsen compactification $\overline{M}_{0,n}$.
\end{theorem}

\begin{proposition}
\label{thm:Hassett}
\cite{hassett2003moduli}
Let 
$\mathbf{a}=(a_1, \ldots, a_n)$ and $\mathbf{c}=(c_1, \ldots, c_n)$ such that $a_i \geq c_i$ for all $n \geq i \geq 1$.  Then, there exist a birational morphism 
$$
\rho_{\mathbf{a},\mathbf{c}}: 
\overline{M}_{0, \mathbf{a}} \longrightarrow
\overline{M}_{0, \mathbf{c}}
$$
which is called the reduction morphism. This morphism is an isomorphism in the locus within $\overline{M}_{0, \mathbf{c}}$
 parametrizing stable curves that are stable with respect to both weights $\mathbf{a}$ and $\mathbf{c}$.
\end{proposition}
Next, we recall the relationship between the compactifications introduced in Theorem \ref{thm:GITcompactification} and Theorem \ref{thm:Hassett}.
\begin{proposition}
\label{prop:KSBAtoGIT}
\cite[Sec 8]{hassett2003moduli}
Given $\mathbf{w} + \epsilon:=(w_1 + \epsilon, \ldots, w_n+ \epsilon)$ were $\epsilon \ll 1$ is a positive rational number and $\mathbf{w}$ as in Theorem \ref{thm:GITcompactification}.
Then, there is a morphism 
$ \varphi_{\mathbf{w}}:
\overline{M}_{0, \mathbf{w} + \epsilon}
\longrightarrow \overline{M}_{0, 
\mathbf{ w}}^{\mathrm{GIT}}$
such that: 
\begin{itemize}
    \item[1.] If there are subsets of indexes $I \subsetneq \{1, \cdots, n\}$ with $\sum_{i\in I} w_i = 1$, then $\varphi_{\mathbf{w}}$ is a blow up of the GIT compactification at the GIT boundary $p_{I,I^{c}}$. 
    \item[2.] If there is no such subset, then $\varphi_{\mathbf{w}}$ is an isomorphism. 
\end{itemize}
\end{proposition}
We remark that we only state the result that we used. It was proved in \cite{young2011moduli} and  \cite{Moon2011} that the morphism in Proposition \ref{prop:KSBAtoGIT} is a  blow-up 
of the ideal sheaf given by
Kirwan’s partial desingularization, see
\cite{kirwan1985partial}.

For the partially labeled case described in Equation \eqref{eq:PointsAffineLine},
we should be careful in using the same weight at the non-labelled points, and keeping the labelled and non-labelled points separated at the degenerations. That is, we consider the rational weights $(b_1, \ldots, b_{n+1}) \in \mathbb{Q}^{n+1}$ such that
$b_1 = \ldots = b_n$,  $b_{n+1} + b_{i} > 1$
$b_1 + \ldots +b_{n+1} > 2$ and 
$1 \geq b_1 > 0$. We have 
$M_{0,n}^* \cong M_{0,n+1}/S_n \subset \overline{M}_{0,\mathbf{b}}/S_n$
with the permutation group $S_n$ acting on the first $n$ labelled points.

\subsection{Eigenperiod map for $n$ points in the line}
\label{sec:hodgetheoryBackground}
Let $\mu := (p_1, \ldots, p_n)$ be a collection of $n$ distinct labelled points in $\mathbb{P}^1$ and let us choose an affine chart such that $p_i= [1:x_i]$ for all $i$. Given a fixed integer $d \geq 2$, we can associated to our collection of points $\mu$ a smooth curve 
$C_{\mu,d}$ of weighted degree $nd$ in the weighted projective plane $\mathbb{P}[d:d:n]$  defined by the closure of the affine curve
\begin{align}\label{eq:CurveFromPoints}
    C_{\mu,d}:= 
    \overline{\{ y^d-(z_0-x_1z_1)...(z_0-x_nz_1) =0\}},
    & &
\deg(z_0) = \deg(z_1) = d, \; \deg(y)=n.
\end{align}
if $d|n$, and 
\begin{align}\label{eq:CurveFromPoints2}
    C_{\mu,d}:= 
    \overline{\{ y^d-z_1(z_0-x_1z_1)...(z_0-x_{n-1}z_1) =0\}},
    & &
\deg(z_0) = \deg(z_1) = d, \; \deg(y)=n.
\end{align}
if $d\nmid n$ and $\mathrm{gcd}(d,n-1)=1$. In the second case, $C_{\mu,d}$ admits a branched point of degree $d$ at the infinity. 

Geometrically, $C_{\mu,d}$ can be realized as the $d^{th}$-cyclic cover of $\mathbb{P}^1$ branched along  $(p_1, \cdots, p_n )$. Then, 
the isomorphism class of the curve is defined up an action of $\SL_2 \times S_n$ on the points, and  by varying the points
$(p_1, \ldots, p_n)$ we obtain a flat family of smooth curves $\mathcal{C}_{\mu, d}\xrightarrow{\pi} M_{0,n}$.
This family has been studied in the context of the moduli of curves, see \cite{fedorchuk2011cyclic}, and it defines a polarized variation of Hodge structure (PVHS) over $M_{0,n}$ such that for a fixed $\mu_0 \in M_{0,n}$, we have 
$\left( H^{1}(C_{\mu_0,d}), F^{\bullet}_{\mu_0}, Q_{\mu_0} \right)$.

Let $A_d$ be the cyclic group generated by  the $d$-th primitive root of unity $\lambda := e^{\frac{2 \pi i}{d}}$. 
This group acts on $C_{\mu,d}$ by $\rho: y\rightarrow \lambda y$ and on $(H^{1}(C_{\mu,d}),Q_{\mu})$ by Hodge isometries via a linear representation 
$\rho^*: A_d \longrightarrow  \Aut(H^{1}(C_{\mu,d}),Q_{\mu})$.
The action of $A_d$ induces an eigenspace decomposition 
\begin{align*}
H^{1}(C_{\mu,d}, \mathbb{C})_{\chi} 
=
\lbrace
x \in 
H^{1}(C_{\mu,d}, \mathbb{C})
\;  | \; 
g \cdot x = \chi(g) x, \; \forall g \in \rho^*(A_d)
\rbrace,
&&
\chi \in \Hom\left(A_d, \C \right)
\end{align*}
Since we can identify the character $\chi$ of $A_d$ with either $\lambda^k$ or its complex conjugated, then we labeled them by 
$k$ within $1 \leq k \leq \lfloor \frac{d}{2}\rfloor$.
We denote as $ H^1_{k}(C_{\mu,d}) $ 
and $ H^1_{\overline{k}}(C_{\mu,d}) $
the eigen-spaces with respect to the eigen-value $\lambda^k$ and $\overline{\lambda^k}$ respectively. We have  the eigenspace decomposition
\begin{align*}
H^1(C_{\mu,d}, \mathbb{C})
=
\bigoplus_{
1\leq k\leq \big\lfloor \frac{d}{2}\big\rfloor    
} \left( 
H^1_{k}(C_{\mu,d},  \mathbb{C}) \bigoplus 
H^1_{d-k}(C_{\mu,d},\mathbb{C}) \right).
\end{align*}
By considering the summands 
$H^{p,q}(C_{\mu,d}, \mathbb{C})$ with  $(p,q)=(1,0)$ or $(0,1)$ in the Hodge decomposition of $H^{1}(C_{\mu,d}, \mathbb{C})$, we obtain the eigenpieces
\begin{align*}
H^{p,q}_k(C_{\mu,d}, \mathbb{C})
&:=
H^{p,q}(C_{\mu,d}, \mathbb{C}) \cap 
H^{1}_k(C_{\mu,d}, \mathbb{C})
\end{align*}
with eigen-Hodge numbers defined as
\begin{align}\label{eq:EigenHodgeNumbers}
\mathbf{h}_k=(h^{p,q}_k),   
&&
h^{p,q}_k = \dim H^{p,q}_k(C_{\mu,d}, \mathbb{C})
\end{align}
To define our eigenperiod space, we consider the flag 
$$
F^{p}_{k} := 
\bigoplus_{s \geq p} H^{s,p}_k(C_{\mu,d}, \mathbb{C})
$$
and denote  as $\mathbf{f}_k= (\dim F_k^{p} \;| \; p=0, 1)$ its dimension vector. Now fix a collection of $n$ distinct points $\mu_0$ in $\mathbb{P}^1$, the Grassmannian
$ \mathrm{Grass}( \mathbf{f}_k, H^{1}_k(C_{\mu_0,d}, \mathbb{C})) $ contains as a non-compact locus  defined as
\begin{align}\label{eq:Eigenspace}
D_{n,d,k}:=
\lbrace
F^{\bullet}_{k} \in \mathrm{Grass}( \mathbf{f}_k, 
H^{1}_k(C_{\mu_0,d}, \mathbb{C})
)
\; \big| \;
Q(F_k^1,F_k^1)=0, \;\;
Q(Cv, \overline{v}) > 0 \; 
\rbrace 
\end{align}
where $C$ is the Weil operator and acts on $H^{p,q}_k(C_{\mu_0,d}, \mathbb{C})$ via $i^{p-q}\cdot \mathrm{Id}$. An entirely analogous definition works for $D_{n,d,\overline{k}}$. 

\begin{definition}
\label{def:EigenPeriod}
For fixed $k$ and $d$ as above, the (universal) eigenperiod map of $n$ points in the line is the holomorphic map 
\begin{align*}
\widehat{\phi}_{n,d,k}:\widehat{M}_{0,n} \longrightarrow D_{n,d,k}
\end{align*}
where $\widehat{M}_{0,n}$ is the universal cover of $M_{0,n}$, $D_{n,d,k}$ is defined by Equation \eqref{eq:Eigenspace}, and the map $\widehat{\phi}_{n,d,k}$ is obtained from analytic continuation. The values of $k$ and $d$ are known as the cover data and denoted as $k/d$. 
\end{definition}
Finally, we need the following explicit basis of $H^1(C_{\mu,d}, \mathbb{C})$.
\begin{lemma}\cite[Chap. 3]{McM13}
For $n$ and $d$ fixed, 
the following holomorphic forms as well as their complex conjugates give a basis of $H^1(C_{\mu,d}, \mathbb{C})$:
\begin{equation}\label{eqnholodiff}
\left\{
\frac{z^{j-1}dz}{y^k} \ \bigg| \  
0 <\frac{j}{n}<\frac{k}{d} < 1, \;
j,k \in \mathbb{N}
\right\}
\end{equation}
Moreover, the eigen-Hodge numbers, defined in Equation \ref{eq:EigenHodgeNumbers}, are given by 
 \begin{align*}
    h^{1,0}_{k}=h^{0,1}_{\bar{k}}=\lceil n(k/d)-1 \rceil, 
   && h^{0,1}_{k}=h^{1,0}_{\bar{k}}=\lceil n(1-k/d)-1 \rceil, 
   &&
 1\leq k\leq \bigg\lfloor \frac{d}{2}\bigg\rfloor .   
\end{align*}
\end{lemma}
To describe the eigenperiod map, let $\mathbb{C}^{ \langle n \rangle}$ be the moduli space of $n$ unlabelled distinct points on $\mathbb{C}$. We need to discuss some aspects on the monodromy representation of the fundamental group of $\mathbb{C}^{ \langle n \rangle}$ which is known as the braid group $B_n$. Let fix $\lambda^{k}$ with $\frac{1}{n} < \frac{k}{d} < 1$
where $\lambda := e^{\frac{2 \pi i}{d}}$ is the $d$-th primitive root of unity,
and let's consider the $\mathbb{Z}[\lambda^k]$- module defined as
$
\Lambda_{n,d,k}:= H^1_k(C_{\mu, d}, \mathbb{Z}[\lambda^k]) .
$
The unitary automorphisms preserving this module form the group
\begin{align}
\mathrm{U}(\Lambda_{n,d,k}) \subset  \mathrm{U}\left( H^1_k(C_{\mu, d}) 
\right) \cong \mathrm{U}(r,s).    
\end{align}
where $\mathrm{U}(r,s)$ is the unitary group with mixed signature $(r,s)$.
The fundamental groups of the moduli spaces introduced in the previous sections are related by the following result.
\begin{lemma}
\label{lemma:FundGroups}
Let $d \geq 2$ and $n \geq 5$ be integers such that gcd$(d,n-1)=1$, then the following diagram commutes
\[
\xymatrix{ 
B_{n-1}:=\pi_{1}\left( \mathbb{C}^{n-1}/S_{n-1} \right)
\ar[d] \ar[r] 
& \mathrm{U}(r,s) \ar[d] \\
\mathrm{Mod}^*_{0,n-1}
:= \pi_1\left(
M_{0,n}/S_{n-1}
\right)
\ar[r]
&  \mathrm{PU}(r,s)
}
\]
If moreover $d | n$, we also have the commutative diagram:
\[
\begin{tikzcd}
B_n:= \pi_1\left( \mathbb{C}^n/S_n \right) \arrow[r] \arrow[d]
& \mathrm{U}(r,s) \arrow[d] \\
\mathrm{Mod}_{0,n} 
:= \pi_1\left(
M_{0,n}/S_n
\right)
\arrow[r]
&  \mathrm{PU}(r,s)
\end{tikzcd}.
\]
\end{lemma}
\begin{proof}
It follows from \cite[(2.5)]{McM13}. The main difference between our presentation and the original one is that in \cite{McM13}, there is sometimes an implicit additional branch point at infinity; for instance, see the construction in \cite[8.5]{McM13}. The presence or absence of this implicit branch point induces the two cases we describe.  Indeed, the diagram exists if there exists a degree-$d$ covering map $C\rightarrow \mathbb{P}^1$ branched along $n$ points with the ramification degree of each point being $d-1$. If gcd$(d,n-1)=1$, such branched cover can be obtained from completing the degree-$d$ cover over $\mathbb{C}$ branched along $n-1$ points as the infinity is automatically branched with the same degree; If we also have $d|n$, then we could also assume all $n$ branched points are finite as the infinity is not branched in this case.
\end{proof}
Therefore, the period map described in 
Definition ~\ref{def:EigenPeriod} induces the  monodromy representation 
$
\pi: B_n 
\longrightarrow
\operatorname{Aut}(H^1(C_{\mu_0, d}))
$
which factors through either the mapping class group 
$\mathrm{Mod}^{*}_{0,n-1}$ or 
$\mathrm{Mod}_{0,n}$.
 By considering the eigenspace $H^1_k(C_{\mu_0, d})$ we obtain the monodromy representation
$$
\pi_k: B_n 
\longrightarrow \mathrm{U}(\Lambda_{n,d,k}) \subset \mathrm{U}(r,s)
$$
whose behavior we describe next. 

\begin{theorem}
\cite[Theorem 7.1]{McM13}
\label{thm:Mcmullen}
The values of $n \geq 5$ 
and $\lambda^k \neq 1$ such that $\mathrm{U}(\Lambda_{n,d,k})$ is a discrete subgroup 
$\mathrm{U}\left(H^1_k(C_{\mu,d}) \right) \cong \mathrm{U}(r,s)$ are those given in Table \ref{tab:ancestorCases} and their complex conjugates. In these cases, either:
\begin{itemize}
    \item $\mathrm{U}(\Lambda_{n,d,k})$ is an arithmetic lattice in $\mathrm{U}(r,s)$, or
    \item We have $k=d/2$ and $\mathrm{U}(\Lambda_{n,d,k}) \cong \mathrm{Sp}_{2g}(\mathbb{Z})$ is a lattice in $\mathrm{Sp}_{2g}(\mathbb{R}) \subset \mathrm{U}(g,g)$ 
\end{itemize}
Moreover, the image $\pi_k(B_n)$ of the braid group  is discrete in all these cases. 
\end{theorem}
The behavior of the monodromy representation described in 
Lemma \ref{lemma:FundGroups} and Theorem \ref{thm:Mcmullen}
implies that we should consider two cases when defining the period map. If $d$ divides $n$, then by any $d$-th branched cover of $\mathbb{C}$ branching over $n$ points can be completed to a branched cover of $\mathbb{P}^1$ over the same points, as the infinity is unramified. On the other hand, if $\gcd(d,n-1)=1$, then a $d$-th branched cover of $\mathbb{C}$ branching along $(n-1)$ points can be completed to a branched cover of $\mathbb{P}^1$ over $n$ points, as the ramification index of infinity is $d$.  


\begin{table}
\caption{
Cases when the image of the braid group $\pi_q(B_n)$ is discrete in $\mathrm{U}(\Lambda_{n,d,k})$ for $n\geq 5$, see Theorem \ref{thm:Mcmullen}.
Here, $n$ denotes the number of points. $k/d$ cover data,  
 We only include the cases with $r,s \geq 1$.
    }
\setlength{\tabcolsep}{10pt}
\renewcommand{\arraystretch}{1.4}    
    \begin{tabular}
    {|>{\centering\arraybackslash}m{0.4cm}|>{\centering\arraybackslash}m{0.8cm}|>
    {\centering\arraybackslash}m{0.8cm}|
    >{\centering\arraybackslash}m{1.5cm}|
    }
     \hline
        $n$ &  $k/d$ & (r,s) & $H(n,d,k)$
    \\  \hline
    5 & $1/4$ & (1,3) &   - 
    \\ 
    5 & $3/10$ & (1,3) &  - 
    \\ 
    5 & $1/3$ & (1,3)  &  2 
    \\ 
    5 & $3/8$ & (1,3)  &  - 
    \\ 
    5 & $2/5$ & (1,2)  &  $\infty$  
    \\ 
    5 & $5/12$ & (2,2) &  - 
    \\ 
    5 & $1/2$ & (2,2)  & - 
    \\ 
    6 & $1/4$ & (1,4) & 3 
    \\ 
    6 & $3/10$ & (1,4) &  - 
    \\ 
    6 & $1/3$ & (1,3) &  3 
    \\ 
    6 & $3/8$ & (2,3) &   $\infty$  
    \\ 
    6 & $5/12$ & (2,3)&   $\infty$  
    \\ 
    6 & $1/2$ & (2,2) &   1   
    \\ 
    7 & $1/6$ & (1,5) &  - 
    \\
    7 & $1/4$ & (1,5) & - 
    \\ 
    7 & $3/10$ & (2,4) &  - 
    \\ 
    7 & $1/3$ & (2,4) & - 
    \\ 
    7 & $3/8$ & (2,4) &  -  
    \\ 
    7 & $5/12$ & (2,4) & - 
    \\ 
    7 & $1/2$ & (3,3) & -   
    \\ 
    8 & $1/6$ & (1,6) &  5 
    \\
    8 & $1/4$ & (1,5) &  5  
    \\ 
    8 & $3/10$ & (2,5)&   $\infty$  
    \\ 
    8 & $1/3$ & (2,5) &  2  
    \\ 
    8 & $3/8$ & (2,4) &  $\infty$  
    \\ 
    8 & $5/12$ & (3,4) &  $\infty$ 
    \\ 
    8 & $1/2$ & (3,3) &  1    
    \\ 
    \hline   
    \end{tabular}
    \begin{tabular}
    {|>{\centering\arraybackslash}m{1.3cm}|>{\centering\arraybackslash}m{0.8cm}|>
    {\centering\arraybackslash}m{3.4cm}|
    >{\centering\arraybackslash}m{1.5cm}|
    }
     \hline
    n &  k/d & (r,s) & $H(n,d,k)$
    \\  \hline
    9 & $1/6$ & (1,7) &  - 
    \\ 
    9 & $1/4$ & (2,6) &  - 
    \\ 
    9 & $3/10$ & (2,6) &  - 
    \\ 
    9 & $1/3$ & (2,5) & 2 
        \\ 
    9 & $5/12$ & (3,5) & -  
    \\ 
    9 & $1/2$ & (4,4) &  -  
    \\ 
    10 & $1/6$ & (1,8) &  -  
    \\
    10 & $1/4$ & (2,7) &   3 
    \\ 
    10 & $3/10$ & (2,6) &  $\infty$
    \\ 
    10 & $1/3$ & (3,6)  &  - 
    \\ 
    10 & $5/12$ & (4,5) &  - 
    \\ 
    10 & $1/2$ & (4,4) &  1 
    \\ 
    11 & $1/6$ & (1,9) &  - 
    \\
    11 & $1/4$ & (2,8)  &  - 
    \\ 
    11 & $1/3$ & (3,7) &  2 
    \\ 
    11 & $5/12$ & (4,6) & -   
    \\ 
    11 & $1/2$ & (5,5) & -  
    \\ 
    12 & $1/6$ & (1,9) & 9   
    \\ 
    12 & $1/4$ & (2,8) & 3  
    \\ 
    12 & $1/3$ & (3,7) &  2 
    \\ 
   12 & $5/12$ & (4,6) &  $\infty$  
    \\ 
   12 & $1/2$ & (5,5)  &  1   
    \\ \hline    
     & $1/6$ & 
$\left(  
\lceil \frac{n}{6}-1 \rceil,  \lceil 
\frac{5n}{6}-1 
\rceil    
\right)$ 
& 5
    \\
$n > 12$    & $1/4$ & 
$\left(  \lceil \frac{n}{4}-1 \rceil,  \lceil 
\frac{3n}{4}-1  \rceil     \right)$    
& 3
    \\
     & $1/3$ &
$\left( \lceil \frac{n}{3}-1 \rceil,  \lceil 
 \frac{2n}{3}-1  \rceil     \right)$  
 & 2
    \\
    & $1/2$ &$\left( \lceil \frac{n}{2}-1 \rceil,  \lceil \frac{n}{2}-1 \rceil    \right)$  
    & 1
 \\   
 & &  &
  \\   \hline 
    \end{tabular}
    \label{tab:ancestorCases}    
\end{table}

\begin{proposition}
\label{prop:eigenperiod}
Given $n,d,k$ as in Table \ref{tab:ancestorCases}, 
the eigenperiod map is defined as
\begin{align*}
M_{0,n}/S_n \xrightarrow{\phi_{n,d,k}}  
\Gamma_{n,d,k} \backslash D_{n,d,k}    
&&
\Gamma_{n,d,k}:= \pi_{d,k}(\mathrm{Mod}_{0,n})
\end{align*}
if $d|n$, and as
\begin{align*}
M_{0,n}/S_{n-1} \xrightarrow{\phi_{n,d,k}}  
\Gamma_{n,d,k}^{} \backslash D_{n,d,k}   
&& 
\Gamma_{n,d,k}^{}:= \pi_{d,k}(\mathrm{Mod}^*_{0,n-1})
\end{align*}
if $d\nmid n$ and $\gcd(d,n-1)=1$, factors through the quotient of a Mumford-Tate subdomain by $\Gamma_{n,d,k}$. In other words, the eigenperiod map is not Hodge-generic. 
\end{proposition}
\begin{proof}
We follow the argument in \cite[Sec 3.1]{gallardo2021geometric}.
$V_{\mathbb{C}}:= H^1_{\chi}\oplus H^1_{\bar{\chi}}$.
$G:=\mathrm{Res}_{\mathbb{Q}(q)/\mathbb{Q}}(\mathrm{SL}(V))\cap \mathrm{Sp}(\tilde{V},Q)$ is the Hodge group of the period map which is a rational form of $\mathrm{U}(r,s)$, therefore, the image of the period map in $\Gamma \backslash \mathrm{Sp}(V)/\mathrm{U}(r+s)$ factors through the (generalized) ball quotient
$\Gamma \backslash \mathrm{U}(r,s)/ \mathrm{U}(r)\times \mathrm{U}(s)$. The case $\gcd(d,n-1)=1$ follows by the discussion within
\cite[Sec 2]{McM13}.
\end{proof}
\begin{remark}
The period map from Proposition 
\ref{prop:eigenperiod} is coming from a $\mathbb{C}$-VHS, and in occasion is convenient to have a period map define over a number field, for more details see \cite[Sec 7]{dolgachev2007moduli}.
 For that purpose,  we may also consider the eigenperiod map 
$\phi_{n,d,k,\overline{k}}$
defined by the Hodge decomposition within
$$
H^1_{k}(C_{\mu,d})\bigoplus H^1_{\bar{k}}(C_{\mu,d}).
$$
In this case, we obtain a $\mathbb{Q}(\mathrm{cos}(\frac{2\pi}{d}))$ - VHS of type $(2d-2,2d-2)$.    \end{remark}

By \cite[(2.5)]{McM13}, we have that
$D_{n,d,k} \cong \text{II}(r,s)\cong \mathrm{Sp}(2r+2s)/\mathrm{U}(r+s)$, and $\Gamma_{n,d,k}$ is  a discrete lattice in $\mathrm{PU}(r,s)$. Moreover, we can compose the map $\phi_{n,d,k}$ with the natural finite map $M_{0,n} \to M_{0,n}/S_m$ to obtain the following result. 
\begin{proposition}
\label{prop:EigenPeriodM0n}
Let $n$, $k$ and $d$ be integers as listed in Table \ref{tab:ancestorCases}.
If $\gcd(d, n-1) = 1$, we have the following commutative diagram
\begin{align*}
\xymatrix{
\widehat{M}_{0,n} \ar[d] \ar[rr]
^{\widehat{\phi}_{n,d,k}}
&& D_{n,d,k}
\ar[d]
\\
M_{0,n} \ar[rr]^{\Phi_{n,d,k}} && 
\Gamma_{n,d,k}\backslash D_{n,d,k}
}
\end{align*}
obtained by composing the period map from Proposition \ref{prop:eigenperiod} with the relevant finite quotient by a symmetric group $S_m$, $m=n$ or $n-1$.
\end{proposition}

\begin{question}
By the work of Deligne-Mostow \cite{DM86}, we know that if $n,d,k$ is equal to either 
$(8,4,1)$ or $(12,6,1)$. Then, the map $\Phi_{n,d,k}$ is finite and dominant. 
What are all the other values of $(n,d,k)$ in Table \ref{tab:ancestorCases} such that 
\[
\dim
\left(
\Phi_{n,d,k}
\left(
M_{0,n}
\right)
\right)
= (n-3)?
\]
\end{question}

\subsection{Eigenspectra of isolated hypersurface singularities}
Let $\mathbf{z} := (z_0, \cdots, z_s)$, 
we recall that a polynomial $f(\mathbf{z})$ is called quasi-homogenuous if there exists a weight $\mathbf{w}:=(w_0,...,w_s)\in \mathbb{Q}^{s+1}$ such that 
\begin{align*}
f(\mathbf{z})=\sum a_{\mathbf{m}}\mathbf{z}^{\mathbf{m}}, 
&&
\mathbf{m}\in \mathbb{Z}_{\geq 0}^{s+1}    
\end{align*}
with $\mathbf{w}\cdot \mathbf{m}$ being a constant for all $\mathbf{m}\neq 0$. 
Therefore, for any $t \in \Delta$, the graph of $(f(\mathbf{z})-t=0)$ defines a hypersurface $\mathcal{X}$ in $\mathbb{C}^{s+2}$
which is a family of algebraic varieties 
$\mathcal{X} \to \Delta$
via the map
$f: \mathbb{C}^{s+1}\rightarrow \mathbb{C}$,
We further assume $X_t:=f^{-1}(t)$ is smooth for any $t\in \Delta^*$, and $X_0$ has an isolated singularity given by $f(\mathbf{z})=0$.

\begin{definition}
The Milnor fiber $\mathcal{F}_f$ associated to the family $\mathcal{X}\rightarrow \Delta$ is defined to be $f_1^{-1}(t_0)\cap B_{\epsilon}(0)$ for 
$0<\epsilon \ll t_0 \ll 1$. Its reduced cohomology $V_f:=\tilde{H}^k(\mathcal{F}_f, \mathbb{Q})$ is called the vanishing cohomology of the degeneration.
\end{definition}

Denote $T$ as the generator of monodromy group around $0\in \Delta$ acting on cohomology of the fibers. By Borel's monodromy theorem, $T$ is quasi-unipotent and admits a decomposition $T=T^{\text{ss}}e^N$ where $T^{\text{ss}}$ is the finite-order semisimple factor and $N$ is the monodromy logarithm. To describe our degeneration, Hodge-theoretically, we need some definitions and notations. 
The following theorems are well-known and
we refer the reader to \cite{KL20} and \cite{PS08} for details.

\begin{theorem}
\label{Thm:CycleFunctor}
Let $\psi_f$ be the nearby cycle functor, and 
$\phi_f$ be the vanishing cycle functor. If we denote
\begin{align*}
 H^k_{\mathrm{lim}}(X_t):=\mathbb{H}^k(X_0, \psi_f\mathbb{Q}_{\mathcal{X}})
 &&
 H^k_{\mathrm{van}}(X_t):=\mathbb{H}^k(X_0, \phi_f\mathbb{Q}_{\mathcal{X}})
\end{align*}
, then it holds that 
\begin{itemize}
\item[1.] The nearby and vanishing cycle functors associate both $H^k_{\mathrm{lim}}(X_t)$ and $H^k_{\mathrm{van}}(X_t)$ natural mixed Hodge structures compatible with $T$, in the sense $T^{\mathrm{ss}}$ acts finitely and $N$ acts as an $(-1,-1)$-morphism.
\item[2.] $H^k_{\mathrm{van}}(X_t)\cong V_f$ and  $(V_f)^{p,q} = 0$ for  $ p+q  <k$ and $k+1 < p+q$. 
\end{itemize}
\end{theorem}

\begin{theorem}\label{vcs}
 There is a sequence known as the Vanishing cycle sequence and given by
\begin{equation}
    ...\rightarrow H^k(X_0)\xrightarrow{}{}H^k_{\mathrm{lim}}(X_t)\xrightarrow{}\bigoplus V_f\xrightarrow{\delta}H^{k+1}(X_0)\rightarrow H^{k+1}_{\mathrm{lim}}(X_t)\rightarrow ...
\end{equation}
which is an exact sequence of mixed Hodge structures. Moreover, the monodromy $T=T^{\text{ss}}e^N$ acts on each term and is compatible with all maps in the sequence. The kernel
\begin{equation}
H^k_{\mathrm{ph}}(X_0):=\mathrm{Ker}\{H^k(X_0)\xrightarrow{\mathrm{}}{}H^k_{\mathrm{lim}}(X_t)\}
\end{equation}
is known as the Phantom cohomology, see \cite{KL20}.
\end{theorem}
\begin{remark}
The phantom cohomology $H^k_{\mathrm{ph}}$ describes the global linear dependence relations of local vanishing cycles.
\end{remark}

\begin{theorem}\label{css}
There is a sequence known as the Clemens-Schmid sequence and given by 
\begin{equation}
    ...\rightarrow H_{2s+2-k}(X_0)\rightarrow H^k(X_0)\xrightarrow{\mathrm{}}{}H^k_{\mathrm{lim}}(X_t)\xrightarrow{N}H^k_{\mathrm{lim}}(X_t) \xrightarrow{}H_{2s-k}(X_0)\rightarrow ...
\end{equation}
which is an exact sequence of mixed Hodge structures, and the monodromy $T=T^{\text{ss}}e^N$ acts on each term and is compatible with all maps in the sequence.
\end{theorem}
To investigate the mixed Hodge structure on the vanishing cohomology $V_f$, we write its Deligne splitting as $V_f=\bigoplus_{p,q}V^{p,q}$.
Suppose there is a finite-order automorphism $\rho$ on $V_f$ compatible with $T^{\text{ss}}$ and $e^N$, then we have a finer splitting of $V_f$ into their eigenspaces:
\begin{equation}
    V_f=\bigoplus_{p,q, \lambda, \eta}V_{f,\lambda, \eta}^{p,q},
\end{equation}
where $e^{2\pi i\lambda}$ and $e^{2\pi i\eta}$
are eigenvalues correspond to 
$T^{\text{ss}}$ and $\rho$, respectively. Notice that we can always choose $\lambda,\eta\in [0,1)\cap \mathbb{Q}$ because the eigenvalues of $T^{\text{ss}}$ and $\rho$ are roots of unity.

\begin{definition}\label{defn:Eigenspectra}
The eigenspectra associated to the vanishing cohomology $V_f$ and the finite automorphism $\rho$ is equal to 
the element in the group algebra $\mathbb{Z}[\mathbb{Q}\times \mathbb{Q}\times \mathbb{Z}]$ defined as
\begin{align*}
    {\sigma_E(f, \rho)}:= \sum_{p,q,\lambda, \eta}m_{p+\lambda, \eta, p+q}[(p+\lambda,\eta,  p+q)],
    &&
    m_{p+\lambda,\eta, p+q}:=\textrm{dim}(V_{f,\lambda, \eta}^{p,q}).
\end{align*}
\end{definition}
This eigenspectra is a natural generalization of Steenbrink's mixed-spectrum, which only contains information about the monodromy operator. 

\section{
Describing the locus with finite monodromy}
\label{sec:PureLocus}

In this section, we prove Proposition \ref{prop:extension} and Theorem \ref{thm:codimensions}. 
We start with some require results on the eigenspectra of the degenerations. 
\subsection{Lemmas on eigenspectra}
Let $\mathbf{z} \in \mathbb{C}^{s+1}$ and $f(\mathbf{z})$ be a quasi-homogenuous polynomial of weight $\mathbf{w} \in \mathbb{Q}^{s+1}_{>0}$ and degree $1$. By \cite[Thm. 1]{Ste77}, the only weights appear in the mixed Hodge structure of vanishing cohomology are $p+q=s$ or $s+1$. Hence the eigenspectra given in Definition \ref{defn:Eigenspectra} can be also written as:
\begin{equation*}
    \sigma_E(f,\rho):= \sum_{p+q=s,\lambda, \eta}m_{p+\lambda, \eta, s}[(p+\lambda,\eta,  s)]+\sum_{p+q=s+1,\lambda, \eta}m_{p+\lambda,\eta, s+1}[(p+\lambda, \eta, s+1)].
\end{equation*}
Since $f(\mathbf{z})$ is quasi-homogeneous, the Jacobian algebra satisfies
\begin{equation*}
\frac{\mathbb{C}[[\mathbf{z}]]}{\mathcal{J}(f)}
:=
\frac{
\mathbb{C}[[\mathbf{z}]]}{<f, \frac{\partial f}{\partial z_0},...,\frac{\partial f}{\partial z_s}>}
\cong 
\frac{\mathbb{C}[\mathbf{z}]}{\mathcal{J}(f)}
\end{equation*}
This is a finite dimensional $\mathbb{C}$-vector space of which basis is a set of monomials that we denote $\operatorname{B}(\mathbf{f})$.
Let  
$
\mathcal{B} = \{ 
\beta \in \mathbb{N}^{s+1}_{\geq 0} \; | \; \mathbf{z}^{\beta} \in \operatorname{B}(\mathbf{f}) \}
$
and set
\begin{equation} \label{eq:lb}
l(\beta):= \sum_{0\leq j\leq s}(\beta_j+1)w_j 
\in \mathbb{Q}.
\end{equation}
We need the holomorphic $(s+1)$-forms
\begin{equation} \omega_{\beta}:=\frac{\mathbf{z}^{\beta}d\mathbf{z}}{(f(\mathbf{z})-1)^{\lceil l(\beta) \rceil}}, \;\; \beta\in \mathcal{B}
\end{equation}
defined on $\mathbb{C}^{s+1}\backslash (Z_f:=\{f(\mathbf{z}) -1 =0 \})$ for our following isomorphism.
\begin{lemma}\label{lemma:residue}
There is an isomorphism:
\begin{equation*}
    H^{s+1}(\mathbb{C}^{s+1}\backslash Z_f)\xrightarrow[\cong]{\mathrm{Res}}H^s(Z_f)\cong V_f
\end{equation*}
and each $\eta_{\beta}:=\mathrm{Res}_{Z_f}(\omega_{\beta})$ satisfies:
\begin{equation*}
\eta_{\beta}\in \bigoplus_{\mu}V^{
\lfloor s+1-l(\vec{\beta})\rfloor,
\lfloor l(\vec{\beta})\rfloor}_{f, \{
s+1-l(\vec{\beta})\}, \eta}, \beta\in \mathcal{B}
\end{equation*}
\end{lemma}
\begin{proof}
The isomorphism is given by the residue map and by \cite[Chap. 2]{KL20}
with the later generated by  
$\mathrm{Res}_{Z_f}([\omega_{\beta}])$.
\end{proof}

\begin{corollary}\label{cor:GrWlb}
The dimension of $\mathrm{Gr}_{s+1}^{W}(V_f)$ equals to the number of $\beta\in \mathcal{B}$ such that $l(\beta)\in \mathbb{Z}$.
\end{corollary}
\begin{proof}
By definition of $V_f$ and Lemma \ref{lemma:residue}
\begin{align*}
\dim \left( \mathrm{Gr}_{s+1}^{W}(V_f)  \right)
= \sum_{p,q, p+q = s+1} \dim(V^{p,q}_f) .  
\end{align*}
For any $\beta\in \mathcal{B}$, $\eta_{\beta}\in \mathrm{Gr}_{\lfloor s+1-l(\vec{\beta})\rfloor + \lfloor  l(\vec{\beta})\rfloor}^{W}(V_f)$, so $\dim \left( \mathrm{Gr}_{s+1}^{W}(V_f)  \right)$ is exactly the number of $\beta\in \mathcal{B}$ such that
\[
s+1 = \lfloor s+1-l(\vec{\beta})\rfloor + \lfloor  l(\vec{\beta})\rfloor. 
\]
We know that either $\lfloor s+1-l(\vec{\beta})\rfloor + \lfloor  l(\vec{\beta})\rfloor = s$ or $s+1$, and the latter happens if and only if $l(\beta) \in \mathbb{Z}$. 
\end{proof}

\begin{lemma}
The eigenspectra associated to the vanishing cohomology of the local degeneration  to the singularity  $f(x,y)=y^d+x^l$ is:
\begin{align*}
    \sigma_E(f,\rho) = \sum_{\frac{a+1}{l}+\frac{k}{d}\in \mathbb{Z}}[(\frac{a+1}{l}+\frac{k}{d}, 2-\frac{a+1}{l}-\frac{k}{d}, 2)]
    + \sum_{\frac{a+1}{l}+\frac{k}{d}\notin \mathbb{Z}}[(\frac{a+1}{l}+\frac{k}{d}, 2-\frac{a+1}{l}-\frac{k}{d}, 1)]    
\end{align*}
where $0\leq a\leq l-2$ and $0\leq k \leq d-1$.
\end{lemma}
\begin{proof}
The monodromy logarithm and its induced LMHS on the vanishing cohomology can be computed by eigenspectra via the local degenerating equation $f(x,y)=y^d+x^l=0$ which is quasi-homogeneous with weight $\mathbf{w}=(\frac{1}{l},\frac{1}{d})$. 
Using Equation \eqref{eq:lb}, we compute $l(\beta)$ for every $\beta$ in the Jacobian algebra 
\[
\frac{\mathbb{C}[\mathbf{z}]}{\mathcal{J}(f)}
=
\{x^ay^r|\; 0\leq a\leq l-2, 0\leq r\leq d-2\}.
\]
By using Lemma \ref{lemma:residue}, we can use the value of $l(\beta)$ to calculate the dimension of $V_{f}^{p,q}$ as in the proof of Corollary \ref{cor:GrWlb}. The dimensions of these graded pieces define the eigenspectra. 
\end{proof}

\begin{corollary}\label{maincorollary}
Let $d, k, l$ be fixed positive integers such that  
$2 \leq l$, $1\leq k \leq d-1$, then we have:
\begin{align*}
h^{1,1}_k\left(  V_f \right) 
&=
\begin{cases}
1 & \text{ if } \;  \text{ $d$ divides $kl$} \\
0 & \text{ if } \;  \text{ otherwise }  
\end{cases}
\end{align*}
\end{corollary}
\begin{proof}
The $h^{1,1}(V_f)$ of the vanishing cohomology $V_f$ is equal to the cardinality of the set:  
$$
\{
\beta:= (a,k) \in \mathbb{N}_+
\; |\; 
l(\beta) = \frac{a+1}{l}+\frac{k}{d}\in \mathbb{Z}\}.
$$
by Corollary \ref{cor:GrWlb}  and the fact that $\mathrm{Gr}^{W}_2(V_f) = H^{1,1}(V_f)$. 
By definition the weighted degree of the monomials in $\mathcal{B}$
is less than or equal to one, so $l(\beta)< 2$. For a fixed $k$, $h^{1,1}_k(V_f)$
is equal to the cardinality of the set
\begin{align*}
\lbrace
a \in \mathbb{N}_{+} \; | \
l-2 \geq a \geq 0 
\;\; \text{ and } \;\;
\frac{a+1}{l}+\frac{k}{d}=1
\rbrace.
\end{align*}
The equation implies
\[
a = 
l-1-\frac{kl}{d},
\]
so the corresponding $\rho$-eigenperiod satisfies $h_k^{1,1}=1$ or $0$, depends on whether $d|kl$ or not.  
\end{proof}

\subsection{Calculation of Limiting Mixed Hodge Structures}
\label{sec:ExtensionPureLocus}
We recall our setting. Let $\mathcal{I}$ be the set indexing irreducible boundary divisors of $\overline{M}_{0,n}$. For $I\in \mathcal{I}$, let $\left( \mathcal{Y}_{I}, \mathcal{P}_I \right) \to \Delta$ be a one-dimensional family of stable rational curves which central fiber corresponds to a point in the boundary divisor $D_I \subset \overline{M}_{0,n}$, and which generic point is contained in $M_{0,n}$. To define such family is equivalent to give a morphism 
$f_{\mathcal{Y}}:\Delta \to \overline{M}_{0,n}$
because $\overline{M}_{0,n}$ is a fine moduli space. Let 
$R: \overline{M}_{0,n} \to \overline{M}^{\mathrm{GIT}}_{0,\mathbf{\frac{2}{n}}}$ be the reduction morphism. We obtain  a morphism
$R\circ f_{\mathcal{Y}}:\Delta \to 
\overline{M}^{\mathrm{GIT}}_{0,\mathbf{\frac{2}{n}}}$ which has associated a unique family 
$\left( \mathcal{Y}'_{I}, \mathcal{P}'_I \right) \to \Delta$ of GIT semistable closed orbits. 
By Section \ref{sec:hodgetheoryBackground}, we have a one-dimensional family of smooth curves $\mathcal{C} \to \Delta$  such that a generic fiber is a smooth $d$-th cover of $\mathbb{P}^1$. 
The central fiber $\mathcal{C}|_0$ has singularities that are locally analytically isomorphic to one defined by the equation $f(x,y)=y^d + x^l$.

\begin{lemma}\label{lemma:OnlyOneSing}
Let $\mathcal{C} \to \Delta$ be a family of curves as above. Suppose that the central fiber has a non-empty set $S$ of multiple singularities 
$ \s \in S$ and let $V_{f,\s}$ be the local vanishing cohomology group around the singular point 
$ \s \in S$. Then, the following inequality holds:
\begin{equation}\label{eq:IneqMultiple}
    \mathrm{Max}\{\mathrm{dim}(V_{f,\s}^{1,1})| \; \s\in S\}
    \leq h^{1,1}_{\mathrm{lim}}(C_t)\leq \sum_{\s \in S}\mathrm{dim}(V_{f,\s}^{1,1}).
\end{equation}
Moreover, the inequality \eqref{eq:IneqMultiple} stays true after restricting to each $\lambda^k$-eigenspace of $H^{1,1}(V_f)$ and $H^{1,1}_{\mathrm{lim}}(C_t)$. 
\end{lemma}
\begin{proof}
By construction the central fiber of $\mathcal{C} \to \Delta$ has singularities of form $y^d + x^{l_i}$ for some $l_i\geq 2$. They appear as the cover of a configuration of points where the marked points $\{p_{j},...,p_{j+l_i-1}\}$ overlap in $\mathbb{P}^1$.
Let $V_{f,s}$ be the local vanishing cohomology around the singular point 
$ \mathfrak{s}\in S$ (and defined in Theorem \ref{Thm:CycleFunctor}.)
In our case, the vanishing cycle sequence \eqref{vcs} is equal to:
\begin{equation}\label{eq:VaniCycleMultiple}
    0\rightarrow 
    H^1(C_0)\xrightarrow{}{}
    H^1_{\mathrm{lim}}(C_t)
    \xrightarrow{}
    \bigoplus_{\s \in S}
    V_{f,\s}\xrightarrow{\delta}H^2_{\mathrm{ph}}(C_0)\rightarrow 0.
\end{equation}
To show the first inequality on the left of \eqref{eq:IneqMultiple}, we use Schmid's nilpotent orbit theorem, polarized relations on nilpotent orbits and associated limiting mixed Hodge structures.  Indeed, we
consider the following $|S|$-dimensional family
\begin{equation}
    \mathcal{\tilde{C}}\rightarrow \Delta^{|S|}, \ z=(z_1,...,z_{|S|})\in \Delta^{|S|}
\end{equation}
with smooth fibers for $z\in (\Delta^{*})^{|S|}$. For $i=1,...,|S|$, the divisor $\{z_i=0\}$ is given by collapsing the $l_i$ points $\{p_{j},...,p_{j+l_i-1}\}$ to a single point. By construction, the central fiber of such a family is given by collapsing $|S|$ tuples of marked points, and the local monodromy group is isomorphic to $\mathbb{Z}^{|S|}$. By \cite[Sec. 3]{KPR19} we obtain the left hand side of \eqref{eq:IneqMultiple} due to the polarized relation of LMHS's summarized in \cite[Sec. 1]{KPR19}. Note that in our weight-$1$ case, the polarized relation on all possible LMHS's is just induced by the values of the $h^{1,1}\in \mathbb{Z}$ and the natural linear relation. 

For the inequality on the right of \eqref{eq:IneqMultiple}, we recall that sequence \eqref{eq:VaniCycleMultiple} is a sequence of MHS, so we count dimensions for the (1,1)-pieces by using additivity of the dimension with respect to that exact sequence.  We obtain
\begin{equation}
    h^{1,1}(C_0) + \sum_{\s \in S}\mathrm{dim}(V_{f,\s}^{1,1})
    = 
     h^{1,1}_{\mathrm{lim}}(C_t) + h^{1,1}_{\mathrm{ph}}(C_0).
\end{equation}
By  Clemens-Schmid in the weight $1$ case, see \eqref{css}, and the fact that $\dim(C_0)=1$ implies $H_{3}(C_0) =0$, we obtain 
\begin{equation*}
0 \xrightarrow[]{}    H^1(C_0)
\xrightarrow{\mathrm{sp}}{}
H^1_{\mathrm{lim}}(C_t)
\xrightarrow{N}H^1_{\mathrm{lim}}(C_t).
\end{equation*}
The fact that $\mathrm{sp}\left(H^1(C_0)\right) = \mathrm{Ker}(N)$ with 
$N:H^{1,1}_{\mathrm{lim}}(C_t) \to H^{0,0}_{\mathrm{lim}}(C_t)$ implies the natural MHS on $H^1(C_0)$ is of Hodge type $(1,0)$, $(0,1)$ and $(0,0)$. Therefore,  $h^{1,1}(C_0) =0$ and  we obtain that
\begin{align*}
\sum_{\s \in S}\mathrm{dim}(V_{f,\s}^{1,1})
    = 
     h^{1,1}_{\mathrm{lim}}(C_t) + h^{1,1}_{\mathrm{ph}}(C_0)
\end{align*}
so our desired inequality in \eqref{eq:IneqMultiple} follows. Notice that  when the singular curve $C_0$ has a connected normalization, 
\begin{equation*}
H^{2}_{\mathrm{ph}}(C_0)^{\vee}\cong \mathrm{coker}\{H_2(C_t)_{\mathrm{lim}}\rightarrow H_2(C_0)\}. 
\end{equation*}
is trivial, hence in this case, we have
\begin{equation*}
    \sum_{\s\in S}\mathrm{dim}(V_{f,\s}^{1,1})
    = 
     h^{1,1}_{\mathrm{lim}}(C_t).
\end{equation*}
To conclude the proof, we recall that the sequence \eqref{eq:VaniCycleMultiple} is an exact sequence of Mixed Hodge Structures, and it holds whenever we restricted to any eigenspace of the character  $\lambda^k$.
Therefore, the inequality  \eqref{eq:IneqMultiple} also holds for every  $\lambda^k$-eigenspace. 
\end{proof}
\begin{lemma}\label{lemma:criteriaPure}
Notation as on Lemma \ref{lemma:OnlyOneSing},  $h_k^{1,1}(V_f) =0$ if and only the LMHS on $H^1_{\mathrm{lim}}(C_t)_k$ is pure. 
\end{lemma}
\begin{proof}
By Lemma \ref{lemma:OnlyOneSing}, we can suppose that the central fiber of the family $\mathcal{C} \to \Delta$ has a unique singularity which is locally modelled by $y^d + x^l=0$.  By Theorem \ref{Thm:CycleFunctor}, we have that $V_{f}^{0,0}=0$.  We recall that the vanishing cycle sequence within Theorem \ref{vcs} is
compatible with the eigenspace decomposition, so we can restrict them to a fixed eigenspace. We obtain
\begin{equation}\label{eq:VaniLMHS}
0 \xrightarrow[]{}    H^1_k(C_0)
\xrightarrow{\mathrm{}}{}
H^1_{\mathrm{lim}}(C_t)_k
\xrightarrow{\mathrm{}} (V_f)_k
\end{equation}
which can be visualized via Hodge-Deligne diagrams as follows. Therefore, the symmetries of a LMHS imply that 
$H^{1,1}_{\mathrm{lim}}(C_t)_k =\{0\}$ if and only if the LMHS on $H^1_{\mathrm{lim}}(C_t)_k$ is pure. 
\begin{figure}[h!]
\begin{tikzpicture}
\begin{scope}[shift={(0,0)}]
    \node at (-2,0.75) {$0$};
    \draw[->] (-1.5, 0.75) -- (-0.5,0.75); 
    \draw[-,line width=1.0pt] (0,0) -- (0,1.5);
    \draw[-,line width=1.0pt] (0,0) -- (1.5,0);
    \fill (0,1) circle (3pt);
    \fill (0,0) circle (3pt);
    \fill (1,0) circle (3pt);
\end{scope}
\begin{scope}[shift={(4,0)}]
    \draw[->] (-2, 0.75) -- (-1,0.75); 
    \node at (-1.5,1) {$\mathrm{}$};    
    \draw[-,line width=1.0pt] (0,0) -- (0,1.5);
    \draw[-,line width=1.0pt] (0,0) -- (1.5,0);
    \fill (0,1) circle (3pt);
    \fill (1,0) circle (3pt);
    \fill (0,0) circle (3pt);
    \fill (1,1) circle (3pt);
    \draw [-stealth](0.9,0.9) -- (0.2,0.2) [line width = 1.0pt];
\end{scope}    
\begin{scope}[shift = {(8,0)}]
    \draw[->] (-2, 0.75) -- (-1,0.75); 
    \node at (-1.5,1) {$\mathrm{}$};    
    \draw[-,line width=1.0pt] (0,0) -- (0,1.5);
    \draw[-,line width=1.0pt] (0,0) -- (1.5,0);
    \fill (0,1) circle (3pt);
    \fill (1,1) circle (3pt);
    \fill (1,0) circle (3pt);
\end{scope}
\end{tikzpicture}
\end{figure}
\end{proof}

\subsection{Proof of Proposition \ref{prop:extension} and 
Theorem \ref{thm:codimensions}
} 
\begin{proof}[Proof of Proposition \ref{prop:extension}]
Let $\Phi_{n,d,k}:M_{0,n} \to 
\Gamma_{n,d,k} \backslash D_{n,d,k}$ be the period map defined on Proposition 
\ref{prop:EigenPeriodM0n}. 
The complement of $M_{0,n}$ in $\overline{M}_{0,n}$ is the union of divisors denoted as $D_{I}$ with $I\subset \{1, \ldots,n\}$, $2 \leq |I| \leq \lfloor \frac{n}{2} \rfloor$. 
For any of such divisors, we consider its interior 
\[
D_{I}^*
:=
D_{I} \setminus \bigcup_{I \neq J} D_{J}
\]
Let $N_{I}$ be the local monodromy of the lifted period map around $D_{I}$. We label the set of divisors with finite monodromy, that is
\[
\mathcal{I}^{\mathrm{pr}} := 
\{
I \; | \;
N_{I} \; \text{ has finite order}
\}. 
\]
Since $\overline{M}_{0,n}$ is a smooth and
normal crossing compactification of $M_{0,n}$, it is a classical result by Griffiths, see  \cite[Theorem 9.5]{Gri70}, that $\Phi_{n,d,k}$ can be extended uniquely and holomorphically to 
\[
M_{n,d,k}^{\mathrm{pr}}:=M_{0,n}
\bigcup (\cup_{I \in \mathcal{I}^{pr}} D_{}^{*})
\subset
\overline{M}_{0,n}
.
\]
This shows Proposition \ref{prop:extension}.
\end{proof}

\begin{remark}
All the Deligne-Mostow cases can be obtained from $n=8$ and $n=12$ which are known as the ancestral cases \cite{doran2004hurwitz}. For the case $n=12$, we don't need to make a finite base change for the Kato-Usui extension \cite[Thm 1.1]{gallardo2021geometric}.
However, for $n=8$ we need to lift to the ordered case by \cite{HM22}.
\end{remark}

\begin{proof}[Proof of Theorem \ref{thm:codimensions}]
\label{sec:ProorCorollary}

Let $\Delta$ be a $1$-dimensional subfamily of $\overline{M}_{ 0, \mathbf{w} }^{\mathrm{GIT}}$ centered at $p$ such that $\Delta^{*}\subset M_{0,n}$.  The local period map $\Phi_{n,d,k}|_{\Delta^{*}}$ is obtained from a complex local system $\mathcal{V}\rightarrow \Delta^{*}$. This induces a $1$-dimensional subfamily of $\overline{M}_{0,n}$ centered at some $\tilde{p}\in \pi^{-1}(p)$ given by $\pi^{*}\mathcal{V}\rightarrow \pi^{-1}\Delta^{*}$. By assumption these two ($\mathbb{C}$)-variation of Hodge structures are isomorphic via $\pi$, which means the local monodromy of $\pi^{*}\mathcal{V}\rightarrow  \pi^{-1}\Delta^{*}$ around $\tilde{p}$ is also finite.  Therefore, the space $\pi(M_{n,d,k}^{\mathrm{pr}})\subset 
\overline{M}_{ 0,  \mathbf{w}  }^{\mathrm{GIT}}$ 
contains exactly all points $p\in 
\overline{M}_{ 0, \mathbf{w} }^{\mathrm{GIT}}$ such that for any one-dimensional family $\Delta$ centered at $p$ with $\Delta^{*}\subset M_{0,n}$, $\mathcal{V}|_{\Delta^{*}}$ has finite monodromy around the center $p$. Next, we compute separately compute the two functions in Theorem \ref{thm:codimensions} separately.

\noindent\textit{Case 1: $d|n$}:  For $\frac{n}{2} \geq |I| \geq 2$, the morphism $\pi: \overline{M}_{0,n} \to \overline{M}_{0,\mathbf{w}}^{\mathrm{GIT}}$
contracts all divisors $D_{I,I^c}$ with $|I|=l$, and their image 
$\pi\left( D_I \right)$ is of codimension $|I|-1 = (l-1)$  in  $\overline{M}_{0,\mathbf{w}}^{\mathrm{GIT}}$ for $l \leq  \lfloor  \frac{n-1}{2}\rfloor$.
Particularly, if $l=\frac{n}{2}$, this implies $D_I$ is contracted to $0$-dimensional boundary strata.
By Lemma \ref{lemma:criteriaPure} and Corollary  \ref{maincorollary}, the LMHS of our family is pure if and only if $d$ does not divide $kl$. To conclude our argument, we need to find the smallest $l$ such that $d$ divides $kl$. Notice that we also need $l\leq \frac{n}{2}$ for the sake of GIT-semistability. We have three possibilities.

\begin{enumerate}
    
\item When gcd$(k,d)<\frac{2d}{n}$, which means there does not exist an $1<l<\frac{n}{2}$ such that $d|kl$. In this case $\pi(M_{n,d,k}^{\mathrm{pr}})=\overline{M}_{0,\mathbf{w}}^{\mathrm{GIT}}$ (and we use the convention dim$(\emptyset)=-\infty$).

\item When gcd$(k,d)>\frac{2d}{n}$, there exists $0<l=\frac{d}{\mathrm{gcd}(k,d)}<\frac{n}{2}$ such that $d|kl$. In this case codim$(\overline{M}_{0,\mathbf{w}}^{\mathrm{GIT}}\backslash \pi(M_{n,d,k}^{\mathrm{pr}}))=\frac{d}{\mathrm{gcd}(k,d)}-1$.

\item When $n$ is even and gcd$(k,d)=\frac{2d}{n}$, $\overline{M}_{0,\mathbf{w}}^{\mathrm{GIT}}\backslash \pi(M_{n,d,k}^{\mathrm{pr}})$ is exactly the set of cusps and has codimension $(n-3)$.

\end{enumerate}

\noindent\textit{Case 2: $d\nmid n$  and $\mathrm{gcd}(d,n-1)=1$}: Using the weight 
\[
(w_1,...,w_{n-1},w_{n})= \left( \frac{1}{n-1}+\epsilon,...,\frac{1}{n-1}+\epsilon, 1-(n -1) \epsilon \right)
\]
with $0< \epsilon \ll \frac{1}{n^2}$, the complement of the image of $M_{0,n}$ in $\overline{M}^{\mathrm{GIT}}_{\mathbf{w}}/S_{n-1}$ 
can be stratified by the partition of $[n-1]$, namely $S_1\sqcup S_2\sqcup ... \sqcup S_r=\{1,2,...,n-1\}$ with $\mathrm{Max}\{|S_1|,...,|S_r|\}\leq 
n-2$.  More precisely, we have the following analog results:
\begin{enumerate}
    
\item When gcd$(k,d)<\frac{d}{n-2}$, which means there does not exist an $1<l\leq (n-2)$ such that $d|kl$. In this case $\pi(M_{n,d,k}^{\mathrm{pr}})=
\overline{M}_{0,\mathbf{w}}^{\mathrm{GIT}}$ (and we use the convention dim$(\emptyset)=-\infty$).

\item When gcd$(k,d)>\frac{d}{n-2}$, there exists $0<l=\frac{d}{\mathrm{gcd}(k,d)}< (n-2)$ such that $d|kl$. In this case codim$(\overline{M}_{0,\mathbf{w}}^{\mathrm{GIT}}\backslash \pi(M_{n,d,k}^{\mathrm{pr}}))=\frac{d}{\mathrm{gcd}(k,d)}-1$.

\item When gcd$(k,d)=\frac{d}{n-2}$, for $l=(n-2)$ it holds that $d | kl$. Then, the only divisors where the LMHS may not be pure is $D_{I}$ with $I^c = \{ i,n \}$. Then, $\pi(D_I)$ parametrizes configurations where the $(n-2)$ points $\{ p_i | i \notin I \}$ collide. This has codimension $(n-3)$. 
\end{enumerate}
\end{proof}

\section{Hodge-theoretic interpretation of exceptional boundary components}

In this section, we focus on the proof of Theorem \ref{thm:KU},
which describes the Hodge theoretical extension of the period map introduced in the previous sections. 

\subsection{Preliminaries and Notation}
Let $\phi_{n,d,k}$ be the eigenperiod map defined in Proposition
\ref{prop:eigenperiod}. 
Let $\Delta \subset \overline{M}_{0,n}$ be a one-parameter smoothing of a point in the boundary. That is $0 \in \Delta$ is a point in $\overline{M}_{0,n} \setminus M_{0,n}$ and  $\Delta^*\subset M_{0,n}$. By the construction on Section \ref{sec:hodgetheoryBackground}, we have a one-dimensional family of generically smooth curves $\mathcal{C}_d \to \Delta$ that 
induces a variation of Hodge Structures. Denote the monodromy of this local VHS as $\langle T \rangle \leq \Gamma_{n,d,k}$, and $N:= \mathrm{log}(T)$.  After a finite base change, we can assume $\langle T\rangle$ is unipotent. 
The local period map $\Delta^* \to  \langle T\rangle \backslash D_{n,d,k} $  extends to a space of nilpotent orbits or limiting mixed Hodge structures (LMHS) associated to $N$, 
see \cite{KU08}. More precisely,
\begin{equation}
\xymatrix{
\Delta  \ar@{->}[r]  
&   \langle T\rangle \backslash D_{n,d,k} \sqcup \langle T\rangle \backslash B(N)
\ar@{->>}[d]
\\
&  \langle T\rangle \backslash D_{n,d,k} \sqcup \langle T\rangle \backslash B^0(N)
}
\end{equation}
where we denote $\check{D}_{n,d,k}$ as the compact dual of $D_{n,d,k}$, and
\begin{align*}
B(N) &
:=e^{\langle N\rangle_{\mathbb{C}}}\backslash \{F^{\bullet}\in \check{D}_{n,d,k} \;|\;(N,F^{\bullet}) \text { is a nilpotent orbit } \},
\\    
B^0(N) &:=\{F(\mathrm{Gr}(W_N))^{\bullet}\in \prod \mathrm{Gr}(D^W_{n,d,k})|(N, F^{\bullet})\in B(N)\}
\end{align*}
with $W_N$ being the Jacobson-Morozov filtration on $\mathrm{Lie}(U(r,s))$ associated to $N$, and $D^W_{n,d,k}$ being the period domain of $N$-polarized pure Hodge structures on graded quotients, and
\[
F^p(\mathrm{Gr}(W_N )) = F^p\cap W_N(H)_{\mathbb C}/F^p\cap W_{N-1}(H)_{\mathbb C}
\]
Where $H$ is the underlying $\mathbb{Q}$-vector space of the Hodge structure.

To interpret such an extension, we recall that 
$B(N)$ parametrizes all $N_{\mathbb{C}}$-nilpotent orbits, or say all limiting Mixed Hodge structures polarized by $N$ up to rescaling, with fixed weight filtration $W_N$, and that the natural map $B(N)\twoheadrightarrow B^0(N)$
sends an LMHS's parametrized by $B(N)$ to their graded quotients. 

Moreover, since we are working with polarized Hodge structures of weight one, the limiting mixed Hodge structure of the central fiber can be described as either pure or of type I , see Figure \ref{lmhsfigure}.
If $T$ is not finite, $N$ gives a type I degeneration. 
In our cases, the type of LMHS is uniquely determined by the weight filtration.
In that case, 
we denote the $\mathbb{Q}$-parabolic subgroup corresponds to the weight filtration as $P_W$,  its unipotent radical $U_W$, the complex parabolic corresponds to a fixed base point 
$F^{\bullet}\in D$ as $P_{F^{\bullet}}$, and the centralizer of $N$ as $Z(N)$.
By \cite[chap. 7]{KP16}, the boundary components parametrizing our LMHS are
\begin{align}\label{hodgebdcomp}
 B(N)
&\cong e^{\langle N\rangle_{\mathbb{C}}} \backslash Z(N)\cap(U_W(\mathbb{C})\rtimes P_W(\mathbb{R}))/Z(N)\cap \mathrm{Stab}_{F^{\bullet}}(\mathbb{C})   
\\
B^0(N)
& \cong B(N)/Z(N)\cap U_W(\mathbb{C})
\notag
\end{align}

\begin{remark}
In \cite{KP16}, by $B(N)$, they meant a single connected component, or $Z(N)$-orbit as $N$ may polarize LMHS of different types. In our weight $1$ case, since $Z(N)$ acts transitively on $B(N)$, therefore the two definitions agree. Examples, where $B(N)$ contains multiple $Z(N)$-orbits parametrizing LMHS with different Hodge numbers, could appear only in higher weights.
\end{remark}

\begin{center}\label{lmhsfigure}
\begin{figure}
\begin{tikzpicture}[scale=0.7]
\begin{scope}[shift ={(0,0)}]
    \draw[-,line width=1.0pt] (0,0) -- (0,3.5);
    \draw[-,line width=1.0pt] (0,0) -- (3.5,0);
\fill (3,0) circle (3pt);
\fill (0,3) circle (3pt);

\node at (3,-0.5) {$r$};
\node at (-0.5,3.5) {$s$};
\end{scope}
\begin{scope}[shift ={(7,0)}]
\draw[-,line width=1.0pt] (0,0) -- (0,3.5);
\draw[-,line width=1.0pt] (0,0) -- (3.5,0);
\fill (0,0) circle (3pt);
\fill (3,0) circle (3pt);
\fill (0,3) circle (3pt);
\fill (3,3) circle (3pt);
\node at (-0.3,-0.3) {$a$};
\node at (3.3,3.3) {$a$};
\node at (-1,3.5) {$s-a$};
\node at (3,-0.5) {$r-a$};
\node at (1.1,1.9) {$N$};
\draw [-stealth](2.8,2.8) -- (0.2,0.2) [line width = 1.0pt];
\end{scope}
\begin{scope}[shift ={(14,1.5)}]
\draw[-,line width=1.0pt] (0,0) -- (5,0);
\draw[-,line width=1.0pt] (2.5,2.5) -- (2.5,-2.5);
\fill (2.5,0) circle (3pt);
\fill (2.5,2) circle (3pt);
\fill (4.5,2) circle (3pt);
\fill (4.5,-2) circle (3pt);
\fill (4.5,0) circle (3pt);
\fill (2.5,-2) circle (3pt);
\fill (0.5,-2) circle (3pt);
\fill (0.5,0) circle (3pt);
\fill (0.5,2) circle (3pt);
\draw [-stealth](2.3,1.8) -- (0.7,0.2) [line width = 1.0pt];
\draw [-stealth](4.3,1.8) -- (2.7,0.2) [line width = 1.0pt];
\draw [-stealth](2.3,-0.2) -- (0.7,-1.8) [line width = 1.0pt];
\draw [-stealth](4.3,-0.2) -- (2.7,-1.8) [line width = 1.0pt];
\end{scope}
\end{tikzpicture}
\centering
\caption{Hodge-Deligne diagram for boundary type zero (left), type I (center) and the adjoint Hodge-Deligne diagram for type I degeneration (right)}
\end{figure}
\end{center}

\subsection{Proof of Theorem
\ref{thm:KU}(i)}\label{sec:ProofThm1.1.ii}

The case when the domain $D_{n,d,k}$ is a complex ball is proved in \cite[Sec. 3.3]{gallardo2021geometric}. 
Their arguments still hold for general signatures $(r,s)$ if we only try to extend the eigenperiod map to the codimension $1$ boundary strata. We describe the proof next.  Choose a neat subgroup $\Gamma^{\mathrm{nm}}_{n,d,k}\leq \Gamma_{n,d,k}$ of finite index. The choice of a neat subgroup induces a finite cover of  
$\overline{M}_{0,n}$ which we normalize and denote as 
$\overline{M}^{\mathrm{nm}}_{0,n}$.  
By construction, $\overline{M}^{\mathrm{nm}}_{0,n}$ has a singular locus of codimension larger than or equal to two. Then, by Kato-Usui's theory \cite[Thm. B]{KU08}, we have a unique extension  to the partial compactification
\begin{align*}
 \Phi_{n,d,k}^{ \mathrm{nm} }:
 \overline{M}^{\mathrm{nm}}_{0,n}
 \dashrightarrow
 \left( \Gamma^{\mathrm{nm}}_{n,d,k}\backslash D_{n,d,k} \right)
  \bigsqcup_{ N\in \Gamma^{\mathrm{nm}}_{n,d,k} \backslash \Sigma}(\Gamma_N\backslash B(N))
\end{align*}
 where the fan $\Sigma$ is constructed by taking the monodromy logarithm around each divisor and their $\mathrm{Ad}(\Gamma^{\mathrm{nm}}_{n,d,k})$-conjugates. This fan is clearly $\Gamma^{\mathrm{nm}}_{n,d,k}$-strongly compatible and thus gives rise to a Kato-Usui type extension to codimension $1$ strata.

\subsection{Proof of Theorem \ref{thm:KU}(ii)}\label{sec:ProofTorelli}
We describe the geometry of the map $\mathcal{R}$. Let $[n]=\{1,2,...,n=2r\}$ and  fix $I\subset [n]$ such that $|I|=r$. Let's recall a generic point $p \in \mathcal{D}^{\circ}
\subset M^{\circ}_{0,\mathbf{\frac{2}{n}} +
\epsilon}$ parametrizes a stable curve $(X, \sum_{i=1}^{2r} \left( \frac{1}{2} + \epsilon \right)x_i)$  such that
$X=X_1\cup X_2\cong \mathbb{P}^1\cup \mathbb{P}^1$ with  $X_1\cap X_2= q$. The points $x_i \in X_1$ for $i \in I$ and $x_i \in X_2$ for $i \in [n]\setminus I$. These points are mutually distinct and away from $X_1 \cap X_2$.
There is a natural reduction map:
\begin{align}\label{eq:ksbareductionmap}
\mathcal{R}:&  \mathcal{D}^{\circ}\rightarrow M_{0,r}/S_r\times M_{0,r}/S_r,
&&
\mathcal{D}^{\circ} \subset
\overline{M}_{0,\frac{2}{n}+\epsilon}^{\circ}
\end{align}
induced by sending the configuration on each component $X_i$ to the corresponding element in $M_{0,r}/S_r$, that is
\begin{align}\label{eq:R}
\left(
X, \sum_{i=1}^{2r} \left( \frac{1}{2} + \epsilon \right)x_i \right) 
\mapsto
\left( 
\mathbb{P}^1, \sum_{i=1}^{r} x_i
\right) , 
\left( \mathbb{P}^1, \sum_{i=r+1}^{2r} x_i \right).
\end{align}
 For any $x\in M_{0,r}/S_r\times M_{0,r}/S_r$, the fiber $\mathcal{R}^{-1}(x)$ is a $2$-dimensional subfamily given by fixing all $x_i$ and varying the gluing point on each component.

\begin{figure}[h!]
\centering
\includegraphics[width=0.5\textwidth]{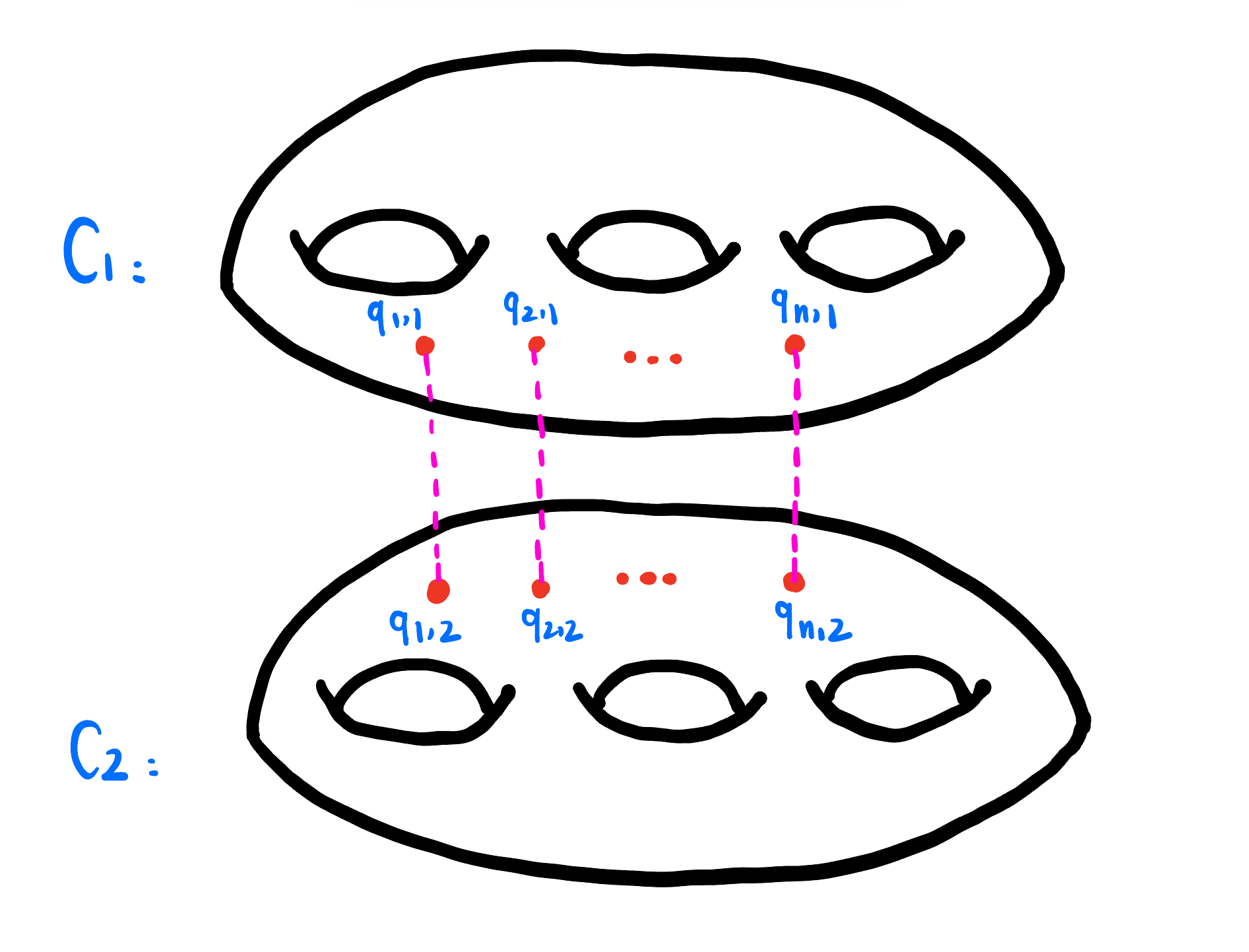}
\caption{Branched cover $C$ and its normalization}
\label{figureksba2}
\end{figure}

Let 
$\Delta\subset \overline{M}_{0,\mathbf{\frac{2}{n}} + \epsilon} $ 
be a one-parameter family such that $\Delta^*\subset M_{0,n}$ and $0\in \mathcal{D}^{\circ}$. As described in Section \ref{sec:hodgetheoryBackground}, when $d|r$ we have a family of curves $\mathcal{C} \to \Delta$ such that $\mathcal{C}|_{t=0}$ is cyclic $d$-cover of the nodal curve $X$ branched along $\{x_1,...,x_n\}$.  Let $C$ be the normalization of $\mathcal{C}|_{t=0}$. By construction, $\mathcal{C}|_{t=0}$ is the union of two components as $C_1\cup C_2$. Both $C_i$ are curves of genus $\frac{(d-1)(r-2)}{2}$ with $C_1\cap C_2=\{q_1,...,q_d\}$ as the set of nodes.
Notice that there are $r$ branched points on each component $X_i$ that define the map $C_i \mapsto X_i$. We identify the (disconnected) curve $C$ with its normalization $C_1\sqcup C_2$ and let $q_{j,1} \in C_1$ and $q_{j,2} \in C_2$ be the points identified with $q_i$ by the normalization map. The divisor 
$q_i:=q_{1,i}+...+q_{d,i}$ in $C_i$ is $\rho$-invariant - recall $\rho$ is the automorphism of $C$ associated to the finite cover $\mathcal{C} \mapsto X$, see Figure \ref{figureksba2}.

By Theorem \ref{prop:extension}.(ii), there exist a finite cover a
 $M^{\mathrm{nm},\circ}_{0,\mathbf{\frac{2}{n}} +
\epsilon} \longrightarrow
M^{\circ}_{0,\mathbf{\frac{2}{n}} +
\epsilon}
$ such that the period map extends to it. 
Let  $\mathcal{D}^{\mathrm{nm},\circ}$ be a connected component of the divisor within 
$M^{\mathrm{nm},\circ}_{0,\mathbf{\frac{2}{n}} + \epsilon}$ that maps to $\mathcal{D}^{\circ}$.
The extended period map $ \phi_{n,d,k}^{\circ}$ is 
defined within $\mathcal{D}^{\mathrm{nm},\circ}$, so we describe it next.

By using the Clemens-Schmid and normalization sequences together, see \cite[Sec 3.b]{Car80}, we obtain
the exact diagram of mixed Hodge structures 
given by Figure \ref{fig:ExactDiagram}.
\begin{figure}[h!]
\centering
\begin{tikzcd}
& & & &0 \arrow{d}\\
& & & &H_1(C_1)_k\bigoplus H_1(C_2)_k\arrow{d} \\
0\arrow{r} & H^1(C)_k \arrow{r} & H^1_{\mathrm{lim}}(C_t)_k \arrow[r, "N"] & H^1_{\mathrm{lim}}(C_t)_k \arrow{r} & H_1(C)_k\arrow{r} \arrow{d} &0\\
& & & & H^{1,1}_{\mathrm{lim}}(C_t)_k\arrow{d}\\
& & & &0
\end{tikzcd} 
\caption{Exact diagram of LMHS associated to $C$}
\label{fig:ExactDiagram}
\end{figure}

The diagram in Figure \ref{fig:ExactDiagram} implies that
\begin{equation}
H_1(C_1)_k\bigoplus H_1(C_2)_k\cong \mathrm{Gr}_1H^1_{\mathrm{lim}}(C_t)
\end{equation}
and since $\mathcal{C}|_{t=0}$ is nodal, $H^{1,1}_{\mathrm{lim}}(C_t)_k$ is described by the Abel-Jacobi map of certain non-closed paths in its normalization $C$ mapped to cycles in $\mathcal{C}|_{t=0}$ via the identification of $q_{j,1}$ and $q_{j,2}$.

Since our result is a local statement that is irrelevant to finite base change, we can directly work on the original map $\mathcal{R}$ locally - instead of the lifted map
$\mathcal{R}^{\mathrm{nm}}$. Denote $N_0$ as the local monodromy logarithm corresponds to $\mathcal{D}$. It suffices to check a local $2$-parameter family:
\begin{equation}
\Phi: \Delta\times \Delta\rightarrow B(N_0)    
\end{equation}
for $\mathbf{N}:=(N_1,N_2)\in \Delta\times \Delta$. Since moving the intersection points $C_1 \cap C_2$ (i.e., the pinched nodes) does not change the normalized curve, we have: 
\begin{equation}
\Phi^0: \Delta\times \Delta\rightarrow  B^0(N_0)
\end{equation}
is constant, and the period map only varies among the extension data with fixed graded quotient, i.e a fiber of the natural map. 
\begin{equation}
  (W_{N_0}=:W_0, F^{\bullet}, N_0)\in B(N_0)\twoheadrightarrow (\mathrm{Gr}(W_{0}), F^{\bullet}, N_0)\in B^0(N_0).
\end{equation}

Take $F_0^{\bullet}$ such that $(N_0, F_0^{\bullet})=\Phi(0,0)$. It then follows that $\mathrm{Img}(\Phi)\subset \pi^{-1}(\Phi^0(0,0))$ where 
$\pi: B(N)\twoheadrightarrow B^0(N)$, 
or more precisely, combine with formula \ref{hodgebdcomp}, we have:
\begin{equation}
\mathrm{Img}(\Phi)\subset e^{\langle N_0\rangle_{\mathbb{C}}} \backslash Z(N_0)\cap U_{W_0}(\mathbb{C}) /Z(N_0)\cap P_{F_0^{\bullet}}(\mathbb{C})
\end{equation}

This implies the map $\Phi$ is totally determined by its image in $I^{-1,0}$, where $\mathrm{Lie(U}(r,s))=\oplus_{p,q}I^{p,q}(=I^{p,q}_{W_0, F_0^{\bullet}})$ is the Deligne splitting of the adjoint diagram \ref{lmhsfigure} (recall that $I^{-1,-1}$ is modded out by $N_0$. $I^{-1,1}$ lies in the Levi subgroup and is irrelevant along any $\pi$-fiber.)

Working on the dual homology sequence, the previous discussion shows the extension class is determined by the image under Abel-Jacobi map of some element in $H^0(\{q_{j,i}\})[-1], i=1,2$. Several things to be noticed are:

\begin{enumerate}
    \item The normalized curve $C$ is not path-connected in our case, therefore the path should be split into $2$ components with endpoints identified under the normalization map;
    \item The "$\rho$-eigen"-extension data should come from (the Abel-Jacobi image of) paths coming from normalizing $\rho$-eigen $1$-cycles.
    \item The standard Abel-Jacobi map only gives a class in the extension data modulo the lattice of integral homology cycles. To define a class in the affine space, we have to find a locally canonical construction.
\end{enumerate}

We fix a component $C_1$. Since any finite cover does not change the local behavior, we consider the local extended period map of the ordered case, i.e., we associate the points a fixed order $\{x_1,...,x_r\}\subset \mathbb{P}^1$ and denote the corresponding branched points in $C$ as $\{p_1,...,p_r\}$. By applying a local $\mathrm{SL}_2$-(constant) section of linear transformation we can assume $(x_1,x_2,x_3)=(0,1,\infty)\in (\mathbb{P}^1)^3$. For the ordered set of pinched points $\{q_{j,1}\}$ with $\rho^l(q_{j,1})=N_{j+l,1}$, $q_{l+d,1}:=q_{l,1}$, we choose a simple path from $0$ to $q_{1,1}$, denoted as $\gamma^1_{q_1}$. If $q_{1,1}$ varies in a contractible neighborhood, we assume the path $\gamma^1_{q_1}$ also varies in a contractible tubular neighborhood. 
Let $\zeta_d=e^{\frac{2\pi i}{d}}$ be the $d$-root of unity and consider the following two elements in the relative homology group $H_1(C_1,\{q_{i,1}\};\mathbb{Z})$:
\begin{center}
$\gamma^1_+:=\gamma^1_{q_1}+\zeta_d^{-1}\rho \gamma^1_{q_1}+...+\zeta_d^{-(d-1)}\rho^{d-1} \gamma^1_{q_1}$
\\

$\gamma^1_-:=\gamma^1_{q_1}+\zeta_d^{}\rho \gamma^1_{q_1}+...+\zeta_d^{(d-1)}\rho^{d-1} \gamma^1_{q_1}$
\end{center}
It is clear $\rho(\gamma^1_{\pm})=\zeta_d^{\pm 1}\gamma^1_{\pm}$. If we do a similar procedure on $C_2$ to obtain $\gamma^2_{\pm}$.
Let $C_1 \sqcup C_2  \to C$ be the normalization map, the image
 of $\gamma^1_{\pm}-\gamma^2_{\pm}$ will give a well-defined element in $H_1(C)_{\pm k}$.

Given $\gamma^1_{\pm}$ and a differential $1$-form $\omega \in H^1(C_1)_k$, by the symmetry in Lemma \ref{eqnholodiff}, we have:
\begin{equation}
    H^1(C_1)_k \xrightarrow[]{\cong} \Omega^1(C_1)_{\pm k},
\end{equation}
where $\omega$ is sent to $\omega$ (resp. $\bar{\omega}$) if $\omega$ is holomorphic (resp. anti-holomorphic). We have the natural map
$\omega \mapsto \int_{\gamma^1_{\pm}} \omega$. 
In the construction of $\gamma^1_{\pm}$,  we choose the point $q_{1,1}$ and a path $\gamma^1_{q_1}$, they define maps 
\begin{equation}
q_{1,1}\xrightarrow[]{\gamma^1_{q_1}} \gamma^1_{\pm} \xrightarrow[]{\tau} \Omega^1(C_1)_{k=\pm 1}^{\vee}.
\end{equation}
For a fixed $q_{1,1}$, if we select two different paths 
$\gamma^{1}_{q_1}$ and $\hat{\gamma}^{1}_{q_1}$ varying in the same contractible tubular neighborhood, the resulting image of $\tau$ in $\Omega^1(C_1)_{k=\pm 1}^{\vee}$ will not be changed.
So to prove the theorem it remains to show $\tau$ is locally injective, because this is the construction that defines the Abel-Jacobi map. 
We only need to prove the following statement: For some $\omega\in \Omega^1(C_1)_{k=\pm 1}$, when varying $q_{1,1}$ in a contractible neighborhood $U_1 \subset C_1$, the function:
\begin{align*}
f_{\omega}: U_1 \longrightarrow \mathbb{C},
&&
f_{\omega}(q_{1,1}):=\int_{\gamma^1_{q_1}}\omega     
\end{align*}
is locally injective. This is because
\begin{equation}
\int_{\gamma_+^1}\omega = \sum_{j=0}^{d-1}\zeta_d^{-j}\int_{\rho^{j-1}\gamma_{q_{1}}^1}\omega = \sum_{j=0}^{d-1}\zeta_d^{-j}\zeta_{d}^{j}\int_{\gamma_{q_1}^1}w = d\int_{\gamma_{q_1}^1}\omega.
\end{equation}
In other words, we need to show for any point $q_{1,1}$ there exists some $\omega\in \Omega^1(C_1)_{k=\pm 1}$ such that $\omega(q_{1,1})\neq 0$. This will imply $f_{\omega}$ is conformal around any $q_{1,1}$.

By Lemma \ref{eqnholodiff}, $\Omega^{1}(C_1)_{-1}$ is popularized by:
\begin{equation}
\frac{z^{j-1}dz}{y^{d-1}} \ , \  
0 <j/r\neq (d-1)/d<1
\end{equation}
In our case where we assume $d|r$, $r\geq 3$ and $k=1$, by checking the formula \cite[(3.2)]{McM13}, the only possible zero for $dz/y^{d-1}\in \Omega^{1}(C_1)_{-1}$ is $\widetilde{\infty}$, the preimage of $\infty$ under the branched covering map $C_1\rightarrow \mathbb{P}^1$, while $\widetilde{\infty}$ is not a zero of the form $z^{\frac{n}{d}-1}dz/y^{d-1}\in \Omega^{1}(C_1)_{-1}$. Therefore, $\Omega^1(C_1)_{k=\pm 1}$ is basepoint-free. We finish the proof of the theorem by showing the same statement on the component $C_2$.

\bibliographystyle{alpha}
\bibliography{references}

\begin{thebibliography}{CMGH23}

\bibitem[AMRT10]{AMRT10}
A.~Ash, David Mumford, M.~Rapoport, and Y.~S Tai.
\newblock {\em Smooth Compactifications of Locally Symmetric Varieties}.
\newblock Cambridge University Press, 2 edition, 2010.

\bibitem[Car80]{Car80}
James~A. Carlson.
\newblock Extensions of mixed hodge structures.
\newblock 1980.

\bibitem[CDKP23]{castor2022remarks}
Ben Castor, Haohua Deng, Matt Kerr, and Gregory Pearlstein.
\newblock Remarks on eigenspectra of isolated singularities.
\newblock {\em Pacific Journal of Mathematics}, 327(1):29--54, 2023.

\bibitem[CGK09]{chen2009pointed}
L~Chen, A~Gibney, and D~Krashen.
\newblock Pointed trees of projective spaces.
\newblock {\em Journal of Algebraic Geometry}, 18(3):477--509, 2009.

\bibitem[CMGH23]{casalaina2023birational}
Sebastian Casalaina-Martin, Samuel Grushevsky, and Klaus Hulek.
\newblock The birational geometry of moduli of cubic surfaces and cubic surfaces with a line.
\newblock {\em arXiv preprint arXiv:2312.15369}, 2023.

\bibitem[DK07]{dolgachev2007moduli}
Igor~V Dolgachev and Shigeyuki Kond{\=o}.
\newblock Moduli of {K}3 surfaces and complex ball quotients.
\newblock In {\em Arithmetic and geometry around hypergeometric functions}, pages 43--100. Springer, 2007.

\bibitem[DM86]{DM86}
P.~Deligne and G.~D. Mostow.
\newblock Monodromy of hypergeometric functions and non-lattice integral monodromy.
\newblock {\em Publications Mathématiques de l'Institut des Hautes Études Scientifiques}, 63:5--89, 1986.

\bibitem[Dor04]{doran2004hurwitz}
Brent~R Doran.
\newblock Hurwitz spaces and moduli spaces as ball quotients via pull-back.
\newblock {\em arXiv preprint math/0404363}, 2004.

\bibitem[Fed11]{fedorchuk2011cyclic}
Maksym Fedorchuk.
\newblock Cyclic covering morphisms on {$\overline{M}_{0, n}$}.
\newblock {\em arXiv preprint arXiv:1105.0655}, 17, 2011.

\bibitem[GKS21]{gallardo2021geometric}
Patricio Gallardo, Matt Kerr, and Luca Schaffler.
\newblock Geometric interpretation of toroidal compactifications of moduli of points in the line and cubic surfaces.
\newblock {\em Advances in Mathematics}, 381:107632, 2021.

\bibitem[GR17]{gallardo2017wonderful}
Patricio Gallardo and Evangelos Routis.
\newblock Wonderful compactifications of the moduli space of points in affine and projective space.
\newblock {\em European Journal of Mathematics}, 3:520--564, 2017.

\bibitem[Gri70]{Gri70}
Phillip~A. Griffiths.
\newblock Periods of integrals on algebraic manifolds, iii. some global differential-geometric properties of the period mapping.
\newblock {\em Institut des Hautes Études Scientifiques. Publications Mathématiques}, 38:125--180, 1970.

\bibitem[Has03]{hassett2003moduli}
Brendan Hassett.
\newblock Moduli spaces of weighted pointed stable curves.
\newblock {\em Advances in Mathematics}, 173(2):316--352, 2003.

\bibitem[HKM24]{hulek2024compactifications}
Klaus Hulek, SHIGEYUKI Kondo, and YOTA Maeda.
\newblock Compactifications of the {E}isenstein ancestral deligne-mostow variety.
\newblock {\em arXiv preprint arXiv:2403.18345}, 2024.

\bibitem[HM22]{HM22}
Klaus Hulek and Yota Maeda.
\newblock Revisiting the moduli space of 8 points on $\mathbb{P}^1$, 2022.

\bibitem[Kir85]{kirwan1985partial}
Frances~Clare Kirwan.
\newblock Partial desingularisations of quotients of nonsingular varieties and their betti numbers.
\newblock {\em Annals of mathematics}, 122(1):41--85, 1985.

\bibitem[KL23]{KL20}
Matt Kerr and Radu Laza.
\newblock Hodge theory of degenerations, (ii): vanishing cohomology and geometric applications, 2023.

\bibitem[KP16]{KP16}
M~Kerr and G~Pearlstein.
\newblock Boundary components of {M}umford–{T}ate domains.
\newblock {\em Duke Math. J.}, 165(4):661--721, 2016.

\bibitem[KPR19]{KPR19}
Matt Kerr, Gregory Pearlstein, and Colleen Robles.
\newblock Polarized relations on horizontal {SL}(2)'s.
\newblock {\em Doc. Math.}, 24:1295--1360, 2019.

\bibitem[KU08]{KU08}
Kazuya Kato and Sampei Usui.
\newblock {\em Classifying Spaces of Degenerating Polarized Hodge Structures}, volume 169 of {\em Annals of Mathematics Studies}.
\newblock Princeton University Press, 2008.

\bibitem[Loo24]{Loo24}
Eduard Looijenga.
\newblock A ball quotient parametrizing trigonal genus 4 curves.
\newblock {\em Nagoya Math. J.}, 254:366--378, 2024.

\bibitem[McM13]{McM13}
Curtis~T McMullen.
\newblock Braid groups and hodge theory.
\newblock {\em Mathematische Annalen}, 355(3):893--946, 2013.

\bibitem[MFK94]{mumford1994geometric}
David Mumford, John Fogarty, and Frances Kirwan.
\newblock {\em Geometric invariant theory}, volume~34.
\newblock Springer Science \& Business Media, 1994.

\bibitem[Moo11]{Moon2011}
H.B. Moon.
\newblock {\em Birational Geometry of Moduli Spaces of Curves of Genus Zero}.
\newblock Thesis, Seoul National University, August 2011.

\bibitem[PS08]{PS08}
Chris~A.M. Peters and Joseph~H.M. Steenbrink.
\newblock {\em Mixed Hodge Structures}, volume~52 of {\em A Series of Modern Surveys in Mathematics}.
\newblock Springer, 2008.

\bibitem[Ste77]{Ste77}
Joseph Steenbrink.
\newblock Intersection form for quasi-homogeneous singularities.
\newblock {\em Compositio Mathematica}, 34(2):211--223, 1977.

\bibitem[YHHB11]{young2011moduli}
{K}iem {Y}oung {H}oon and {M}oon {H}an {B}om.
\newblock Moduli spaces of weighted pointed stable rational curves via {GIT}.
\newblock {\em Osaka Journal of Mathematics}, 48(4):1115--1140, 2011.

\end{thebibliography}

\end{document}